\documentclass[11pt]{amsart}
\usepackage{latexsym}
\usepackage{amsfonts}
\usepackage{amsmath,amsthm,amssymb}

\newtheorem{proposition}{Proposition}[section]
\newtheorem{theorem}[proposition]{Theorem}

\newtheorem{thm}{Theorem}[section]
\newtheorem{lem}[thm]{Lemma}
\newtheorem{cor}[thm]{Corollary}

\newtheorem{prop}[thm]{Proposition}
\theoremstyle{definition}
\newtheorem{defn}[thm]{Definition}

\newtheorem{obs}[thm]{Observation}

\newcommand{\comment}[1]{}

\begin{document}

\author[Rodney G. Downey]{Rodney G. Downey}

\address{\tt School of Mathematical and Computing Sciences, Victoria
  University, P.O. Box 600, Wellington, New Zealand} \email{\tt
  Rod.Downey@vuw.ac.nz}

\author[Carl G. Jockusch, Jr.]{Carl G. Jockusch, Jr.}

\address{\tt Department of Mathematics, University of Illinois at
  Urbana-Champaign, 1409 West Green Street, Urbana, IL 61801, USA
  \newline http://www.math.uiuc.edu/\~{}jockusch/} \email{\tt
  jockusch@math.uiuc.edu}

\author[Paul E. Schupp]{Paul E. Schupp}

\address{\tt Department of Mathematics, University of Illinois at
  Urbana-Champaign, 1409 West Green Street, Urbana, IL 61801, USA}
  \email{\tt schupp@math.uiuc.edu}

\thanks{Downey was supported by the Marsden Fund of New Zealand, and 
by the University of Chicago, and the Isaac Newton Institute 
at Cambridge for portions of this paper.}

\title[Asymptotic density and computably enumerable sets]{Asymptotic
  density and computably enumerable sets}

\begin{abstract}
  We study connections between classical asymptotic density,
  computability and computable enumerability.  In an earlier paper,
  the second two authors proved that there is a computably enumerable
  set $A$ of density $1$ with no computable subset of density $1$.  In
  the current paper, we extend this result in three different ways:
  (i) The degrees of such sets $A$ are precisely the nonlow c.e.\
  degrees.  (ii) There is a c.e.\ set $A$ of density $1$ with no
  computable subset of nonzero density.  (iii) There is a c.e.\ set
  $A$ of density $1$ such that every subset of $A$ of density $1$ is
  of high degree.  We also study the extent to which c.e.\ sets $A$
  can be approximated by their computable subsets $B$ in the sense
  that $A \setminus B$ has small density.  There is a very close
  connection between the computational complexity of a set and the
  arithmetical complexity of its density and we characterize the lower
  densities, upper densities and densities of both computable and
  computably enumerable sets.  We also study the notion of
  ``computable at density $r$'' where $r$ is a real in the unit
  interval.  Finally, we study connections between density and
  classical smallness notions such as immunity, hyperimmunity, and
  cohesiveness.
\end{abstract}

\maketitle


\section{Introduction}\label{intro}
Perhaps the first \emph{explicit} realization of the basic importance
of computational questions in studying mathematical structures was the
1911 paper of Max Dehn ~\cite{De} which defined the word, conjugacy
and isomorphism problems for finitely generated groups.  Although
mathematicians had always been concerned with algorithmic procedures,
it was the introduction of non-computable methods such as the proof of
the Hilbert Basis Theorem which brought effective procedures into
focus, for example in the early work of Grete Hermann
\cite{Herm1926}. The basic work of Turing, Kleene and Church and
others in the 1930's gave methods of demonstrating the
\emph{non-}computability of problems.  We have since seen famous
examples of non-computable aspects of mathematics such as the
unsolvability of Hilbert's 10th problem, the Novikov-Boone proof of
the undecidability of the word and conjugacy problems for finitely
presented groups, and other similar questions in topology, Julia sets,
ergodic theory etc.  With the advent of actual computers, the late
20th century saw the development of computational complexity theory,
clarifying the notion of \emph{feasible} computations using measures
such as polynomial time.

The methods mentioned  in the paragraph above
are all  \emph{worst case} in that they  focus on  the difficulty
of the hardest instances of a given problem.
\emph{However}, there has been a growing realization that the 
worst case may have  very little to do with the behavior
of  several computational procedures widely used in  practice.
The most famous   example is Dantzig's Simplex Algorithm which 
always runs quickly in practice.   There are examples
of Klee and Minty \cite{KM} which force the Simplex Algorithm
to take exponential time but such examples never occur in practice.

   Gurevich \cite{Gurevich} and Levin \cite{Levin} introduced
\emph{average-case} complexity.  The general
idea is that one has a probability measure on the set of possible instances
of a problem and one integrates the time required over all the instances.
Blass and Gurevich \cite{BG} showed that the Bounded Product Problem
for the modular group, an {\sc NP}-complete problem,  has polynomial
average-case complexity.  There has recently been much interest
in \emph{smoothed analysis} introduced by Spielman and Tang \cite{ST}.
This is a very  sophisticated approach which continuously interpolates
between worst-case and average-case and measures the performance
of an algorithm under small Gaussian  perturbations of arbitrary inputs.
They show that the Simplex Algorithm with a certain pivot rule 
has polynomial time smoothed complexity.

Average-case analysis is highly dependent on the the probability
distribution and, of course, one must still consider the behavior of
the hardest instances.  The idea of \emph{generic case complexity} was
introduced by Kapovich, Myasnikov, Schupp and Shpilrain \cite{KMSS}.
Here one considers \emph{partial algorithms} which give no incorrect
answers and where the collection of inputs where the algorithm fails
to converge is ``negligible'' in the sense that it has asymptotic
density $0$.  (A formal definition is given below.)  This complexity
measure is widely applicable and is much easier to apply.  In
particular, one can consider all partial algorithms, the natural
setting for computability theory, and one does not need to know the
worst-case complexity.  Indeed, undecidable problems can have very low
generic-case complexity.  This idea has been very effectively applied
to a number of problems in combinatorial group theory. To cite two
examples \cite{KS1,KSS}, it is a classic result of Magnus, proved in
the 1930's, that the word problem for one-relator groups is solvable.
We do not have any idea of possible worst-case complexities over the
whole class of one-relator groups, but for any one-relator group with
at least three generators, the word problem is strongly generically
linear time.  Also, we do not know whether or not the isomorphism
problem for one-relator presentations is solvable, but it is strongly
generically single exponential time.  To take an undecidable problem,
it is easy to show that the Post Correspondence Problem is generically
linear time.

The phenomenon that problems are generically easy is very widespread.
For example, we know that {\sc Sat}-solvers work extremely well in
practice in spite of the fact that {\sc SAT} is {\sc NP}-complete.
This certainly reflects the problem being generically easy but we need
a more detailed understanding.  Gaspers and Szeider \cite{GaSz}
suggest that the \emph{parameterized complexity} of Downey and
Fellows \cite{DF} might provide an explanation.

 The paper of Jockusch and Schupp  \cite{JS} was the  first step  
towards developing a general theory of generic computability and
the present paper is a significant 
extension of their work. As we shall  see, there are deep and unexpected 
connections between  ideas from generic computation 
and concepts  from classical computability.

Here are the fundamental 
definitions.

\begin{defn} Let $S \subseteq \omega$, where $\omega = \{0, 1,
  \dots\}$ is the set of all natural numbers.  For every $n > 0$ let
  $S\upharpoonright n$ denote the set of all $s \in S$ with $s < n$.  For
$n > 0$, let

     \[ \rho_n(S) := \frac{|S \upharpoonright n|}{n}\]

The \emph{upper density} $\overline{\rho}(S)$ of $S$ is

\[
\ \overline{\rho}(S) := \limsup_{n \to \infty} \rho_n(S)
\]

and  the \emph{lower density} $\underline{\rho}(S)$ of $S$ is

\[
 \underline{\rho}(S) := \liminf_{n \to \infty} \rho_n(S)
\]

If the actual limit  $\rho(S) =lim_{n \to \infty} \rho_n(S)$   exists, then $ \rho (S)$  is the \emph{(asymptotic) density} of $S$.

\end{defn}

Of course, density is finitely additive but not countably additive.
A set $A$ is called \emph{generically computable}
if there is a partial computable function $\varphi$ such that for all
$n$, if $\varphi(n) \downarrow$ then 
$\varphi(n) = A(n)$, and the domain of $\varphi$ has density $1$.

Jockusch and Schupp \cite{JS} observe that
every nonzero Turing degree contains both a set that is generically
computable  and one that is not.  They also
introduce the related notion of being \emph{coarsely computable}.
A set  $A$ is \emph{coarsely computable} if there is a  \emph{total}
computable function $f$ such that $\{n : f(n) = A(n)\}$ has density
$1$.  They show that  there are c.e.\ sets that are coarsely computable 
but not generically computable  and that there are 
c.e.\ sets which  are generically computable  but not coarsely computable
and give a number of basic properties of these ideas.

  Every finitely generated group has a coarsely computable word
problem (\cite{JS}, Observation 2.14), but it is a
difficult open question as to whether or not there is a finitely presented group
which does not have a generically decidable word problem.  There is a
finitely presented semigroup which has a generically undecidable word
problem by a theorem of Myasnikov and Rybalov \cite{MR}, and 
Myasnikov and Osin \cite{MO}  have shown that there is
a finitely generated, recursively presented group with a generically
undecidable word problem.  (Here the notion of density is defined for
sets of words on a finite alphabet in the natural way.)

Our starting point here is the basic observation that
the domain of a generic decision algorithm is a c.e.\ set of density 1.
This observation leads us to concentrate upon 
the relationship between \emph{density}, \emph{computability},
and \emph{computable enumerability}. 
    One basic question which we address is to what what extent a c.e.\ set $A$ can be approximated by a \emph{computable}
subset $B$ so that the difference $A \setminus  B$  has ``small'' density in various senses.

For example, it is natural to ask whether every c.e.\ set of density 1
has a computable subset of density 1. 
Jockusch and Schupp \cite{JS}   established that the answer is ``no'':
There is a c.e.\ set of density $1$ with no computable subset of 
density $1$.
In this paper we extend this result in several ways, 
revealing a deep connection between notions from classical 
computability theory and generic computation.

The natural question to ask is ``what kinds of c.e.\ sets \emph{do} have a 
computable subset of density 1?''
The answer lies in the the complexity of the sets as measured by their
information content. The reader should recall
that the natural operation on Turing degrees is the 
\emph{jump} operator, the relativization of the halting problem, where
the jump $A'$ of a set $A$ is given by
$A'=\{n : \Phi^A_n(n)\downarrow\}$.   The jump operation on sets
naturally induces the jump operation on degrees. 
This  operator is not injective, and we call sets 
$A$ with $A'\equiv_T \emptyset'$ \emph{low}, and the degrees of low sets
are also called \emph{low}.  Low sets resemble computable sets modulo
the jump operator and share some of the properties of computable sets.   They occupy a central role 
in classical computability. On the other hand, there are almost no 
known natural properties of the c.e.\ sets which occur in exactly the low c.e.\ degrees.

We introduce a new nonuniform technique to prove the following.

\begin{thm} \label{nonlow1} A c.e.\ degree  ${\bf a}$ is not low
  if and only if  it contains a c.e.\  set $A$ of density $1$ with no computable
  subset of density $1$.
\end{thm}

The technique we introduce for handling non-lowness is quite flexible, and 
we illustrate this fact with some easy applications. 
For example,
recall that if  $A$ is a c.e.\ set,  its complement $\overline{A}$ is
called \emph{semilow} if $\{e:W_e\cap \overline{A}\ne \emptyset\}\le_T
\emptyset'$, and is called \emph{semilow$_{1.5}$} if $\{e:|W_e \cap
\overline{A}|=\infty\} \le_m \{e : |W_e| = \infty\}$.  The
implications low implies semilow implies semilow$_{1.5}$ hold, and it
can be shown that they cannot be reversed.
These notions were introduced by Soare 
\cite{soare77}, and Maass \cite{maass} in connection 
with both computational complexity and the lattice of 
computably enumerable sets. 

We also prove the following characterization of non-lowness.

\begin{theorem} If  $\bf a$ is  a c.e.\ degree    then  {\bf a} is not low
if and only if  there is a c.e.\ set $A$ of degree $\bf a$ such that $\overline{A}$
is not semilow$_{1.5}$.
\end{theorem}

We remark in passing that our technique has also found applications in 
effective algebra. Downey and Melnikov 
\cite{DoMe} use the methodology to 
characterize the $\Delta_2^0$-categorical 
homogeneous completely decomposable torsion-free abelian groups
in terms of the semilowness of the type sequence.

Another direction  is to ask what kinds of densities are
guaranteed for computable or c.e.\ subsets.  We prove the following.

\begin{theorem} There is a c.e.\ set of density 1 with 
no computable subset of \emph{nonzero} density. Such sets exist in each 
non-low c.e.\ degree.
\end{theorem} 

This result stands in contrast to the low case, where we show that
all possible densities for computable subsets are achieved.

\begin{theorem} If $A$ is c.e.\ and low and has density $r$,
then for any $\Delta^0_2$ real $\hat{r}$ with $0 \leq \hat{r} \leq r$,
$A$ has a computable subset of density $\hat{r}.$
\end{theorem}

Finally with Eric Astor we prove the following.

\begin{theorem}[with Astor] There is a c.e.\ set $A$ of density $1$ 
such that the degrees of subsets of $A$ of density $1$ 
are exactly the high degrees.
\end{theorem}

On the other hand,  we obtain a number of positive results on approximating
c.e.\ sets by computable subsets.   

\begin{theorem} If $A$ is a c.e.\ set, then for every real number
  $\epsilon > 0$ there is a computable set $B \subseteq A$ such that
  $\underline{\rho}(B) > \underline{\rho}(A) - \epsilon$.
\end{theorem}

It turns out that there is a very close correlation between the
complexity of a set and the complexity, as real numbers, of its
densities.  We measure the complexity of a real number by classifying
its upper or lower cut in the rationals in the arithmetical hierarchy.
 
In \cite{JS}, Theorem 2.21, it was shown that the densities of
computable sets are exactly the $\Delta^0_2$ reals in the interval
$[0,1]$.  In this article we characterize the densities of the c.e.\
sets and the upper and lower densities of both computable and c.e.\
sets.  We assume that we have fixed a computable bijection between the
natural numbers and the rational numbers.  We thus say that a set of
rational numbers is $\Sigma_n$ if the corresponding set of natural
numbers is $\Sigma_n$, and similarly for other classes in the
arithmetic hierarchy.  The following definition is fundamental and
standard:

\begin{defn} \label{left} A real number $r$ is \emph{left}-$\Sigma^0_n$ if its
  corresponding lower cut in the rationals, $\{q \in \mathbb{Q} : q <
  r\}$, is $\Sigma^0_n$.   We define ``left-$\Pi^0_n$'' analogously.
\end{defn}
  
The following result, together with Theorem 2.21 of \cite{JS},
characterizes the densities and the upper and lower densities of the
computable and c.e.\ sets.  It can be easily extended by
relativization and dualization to characterize the densities and upper
and lower densities of the $\Sigma^0_n$, $\Pi^0_n$ and $\Delta^0_n$
sets for all $n \geq 0$.

\begin{thm} \label{i3}
  Let $r$ be a real number in the interval $[0,1]$.  Then the
  following hold:
\begin{itemize}
     \item[(i)] $r$ is the lower density of some computable set  if and only if   $r$ is left-$\Sigma^0_2$.
    \item[(ii)] $r$ is the upper density of some computable set  if and only if   $r$ is left-$\Pi^0_2$.
    \item[(iii)] $r$ is the lower density of some c.e.\ set  if and only if   $r$ is left-$\Sigma^0_3$.
    \item[(iv)] $r$ is the upper density of some c.e.\ set  if and only if   $r$ is left-$\Pi^0_2$.
    \item[(v)]  $r$ is the  density of some c.e.\ set  if and only if   $r$ is left-$\Pi^0_2$.

\end{itemize}
\end{thm}

We also explore the relationship between coarse computability and 
generic computability. The proof that there is a generically computable c.e.\ 
set that is not coarsely computable strongly 
resembles the proof that there is a density 1 c.e.\ 
set without a computable subset of density 1. Thus we might 
expect a similar characterization of the low degrees
using coarse computability but  
here we find  a surprise.

\begin{theorem} Every nonzero c.e.\ degree contains a 
c.e.\ set that is generically computable but not coarsely computable.
\end{theorem}

We also discuss the relationship 
of our concepts with classical smallness concepts 
such as immunity, hyperimmunity, and cohesiveness.   In particular we study the extent
to which various immunity properties imply that a set has
small upper or lower density in various senses.   We show that the results for many standard
immunity properties are different, thus again bringing out the connection between density
and computability theory.  

We also begin the study of sets computable at  density $r < 1$.

\begin{defn}   Let $A \subseteq \omega$ be a set and let $r$ be a
real number in the unit interval.    Then $A$ is \emph{computable at density} $r$ if there is
a partial computable function $\varphi$ such that $\varphi(n) = A(n)$ for
all $n$ in the domain of $\varphi$ and the domain of $\varphi$ has lower
density greater than or equal to  $r$. 
\end{defn}
 
  Note that $A$ is generically computable if and
only if $A$ is computable at density $1$.
We will make the following easy observation.
 
\begin{obs} Every nonzero Turing degree contains  a set $A$ which is computable at every  density 
 $r < 1$ but which is not generically computable.
\end{obs}

Earlier phases of our work included open questions which were subsequently resolved by Igusa \cite{I}
and by Bienvenu, Day, and H\"olzl \cite{BDH}.  In the final two sections we state their surprising
and beautiful results and mention some related work.

\section{Terminology and notation}

As usual, we let $\varphi_e$ be the $e$th partial computable function in a
fixed standard enumeration, and we let $W_e$ be the domain of $\varphi_e$.
We write $\Phi_e$ for the $e$th Turing functional.

As in \cite{JS}, Definition 2.5, define:
$$R_k = \{m : 2^k \mid m \ \& \ 2^{k+1} \nmid m\}$$
Note that the sets $R_k$ are pairwise disjoint, uniformly computable
sets of positive density, and the union of these sets is $\omega
\setminus \{0\}$.    These sets were used frequently in \cite{JS}
and we will also use them several times in this paper.

    \section{Approximating c.e.\ sets by computable subsets}

We consider the extent to which it is true that every c.e.\ set $A$ has
a computable subset $B$ which is almost as large as $A$.  More
precisely, we require that the difference $A \setminus B$ should have
small density.  Here, ``density'' may refer to either upper or lower
density, and ``small'' may mean $0$ or less than a given positive real
number.  For convenience in stating results in this area, we introduce
the following notation, which is not restricted to the case $B \subseteq A$.

\begin{defn}  Let $A, B \subseteq \omega$.
    \begin{itemize}
        \item[(i)] Let $d(A,B)$ be the lower density of the symmetric 
        difference
      of $A$ and $B$ (so $d(A,B) = \underline{\rho}(A \triangle B)$).
    \item[(ii)] Let $D(A,B)$ be the upper density of the symmetric difference of
      $A$ and $B$ (so $D(A,B) = \overline{\rho}(A \triangle B)$).
     \end{itemize}
\end{defn}

  Intuitively, $d(A,B)$ is small if there are infinitely
many initial segments of the natural numbers on which $A$ and $B$
disagree on only a small proportion of numbers, and $D(A,B)$ is small
if there are cofinitely many such initial segments.

The following easy proposition lists some basic properties of $d$ and $D$.

\begin{prop}
    Let $A$, $B$, and $C$ be subsets of $\omega$.
    \begin{itemize}
       \item[(i)] $0 \leq d(A,B) \leq D(A,B) \leq 1$
        \item[(ii)] (Triangle Inequality)  $D(A,C) \leq D(A,B) + D(B,C)$
      \end{itemize}
\end{prop}

Since $D(A,B) = D(B,A)$ and $D(A,A) = 0$ for all $A, B$, it follows
from the above proposition that $D$ is a pseudometric on Cantor space.
Recall (\cite{JS}, Definition 2.12) that two sets are
\emph{generically similar} if their symmetric difference has density
$0$.  Thus $D$ is a metric on the space of equivalence classes of sets
modulo generic similarity.  On the other hand, the triangle inequality
fails for $d$.  For example, if $A$ is any set with lower density $0$
and upper density $1$, we have $d(\emptyset, A) = 0, d(A, \omega) =
0$, and $d(\emptyset, \omega) = 1$.

The following elementary lemma gives upper bounds for $d(A,B)$ and
$D(A,B)$ in terms of the upper and lower densities of $A$ and $B$ in
the case where $B \subseteq A$.   

\begin{lem} \label{diff}
Let $A$ and $B$ be sets such that $B \subseteq A$.
     \begin{itemize}
         \item[(i)] $d(A,B) \leq \overline{\rho}(A) - \overline{\rho}(B)$
         \item[(ii)] $d(A,B) \leq \underline{\rho}(A) - \underline{\rho}(B)$
         \item[(iii)] $D(A,B) \leq \overline{\rho}(A) - \underline{\rho}(B)$
      \end{itemize}
\end{lem}

\begin{proof}
Since $B \subseteq A$, we have that $A \triangle B = A \setminus B$, and hence
$$\rho_n (A \triangle B) = \rho_n(A) - \rho_n(B)$$
for all $n$.  The lemma follows in a straightforward way from the
above equation and the definitions of upper and lower density.  For
example, to prove the first part, let a real number $\epsilon > 0$ be
given.  Let 
$$I = \{n : \rho_n(A) \leq \overline{\rho}(A) + \epsilon/2 \quad \& \quad
\rho_n(B) \geq \overline{\rho}(B) - \epsilon/2 \}$$ 
Then $I$ is infinite because the first inequality in its definition
holds for all sufficiently large $n$, and the second inequality in its
definition holds for infinitely many $n$.  By subtracting these inequalities,
we see that
$$\rho_n(A \setminus B) = \rho_n(A) - \rho_n(B) \leq \overline{\rho}(A) -
\overline{\rho}(B) + \epsilon$$ holds for infinitely many $n$.  Since
$\epsilon > 0$ was arbitrary, we conclude that
$$d(A,B) = \liminf_n \rho_n(A \setminus B) \leq \overline{\rho}(A) -
\overline{\rho}(B)$$
The other parts are proved similarly and are left to the reader.
\end{proof}

We begin with a result of Barzdin' from 1970 showing that every c.e.\ set
can be well approximated by a computable subset on infinitely many
intervals.   We thank Evgeny Gordon for bringing Barzdin's work to 
our attention.

\begin{thm} (Barzdin' \cite{B}) \label{Barzdin} For every c.e.\ set
  $A$ and real number $\epsilon > 0$, there is a computable set $B
  \subseteq A$ such that $\overline{\rho}(B) > \overline{\rho}(A) -
  \epsilon$, and hence (by Lemma \ref{diff}) $d(A,B) < \epsilon$.
\end{thm}

\begin{proof}
  Given such $A$ and $\epsilon$, let $q$ be a rational number such
  that $\overline{\rho}(A) - \epsilon < q < \overline{\rho}(A)$, and
  let $\{A_s\}$ be a computable enumeration of $A$.  We now define two
  computable sequences $\{s_n\}_{n \in \omega}$, $\{t_n\}_{n \in
    \omega}$ simultaneously by recursion.  Let $s_0 = t_0 = 0$.  Given
  $s_n$ and $t_n$ let $(s_{n+1}, t_{n+1})$ be the first pair $(s,t)$
  such that $s > s_n$ and $\rho_s (A_t \setminus [0,s_n)) \geq q$.
  Such a pair exists because $q < \overline{\rho}(A) =
  \overline{\rho}(A \setminus [0,s_n))$, so there are infinitely many
  $s$ with $\rho_s(A \setminus [0,s_n) \geq q$.  Now, for each $x$,
  put $x$ into $B$ if and only if $x \in A_{t_{n+1}}$, where $n$ is
  the unique number such that $x$ belongs to the interval $[s_n,
  s_{n+1})$.  Note that
$$\rho_{s_{n+1}}(B) \geq \rho_{s_{n+1}}(A_{t_{n+1}} \setminus [0,s_n)) \geq q$$
for all $n$.   It follows that $\overline{\rho}(B) \geq q > 
\overline{\rho}(A) - \epsilon$, as needed to complete the proof.
\end{proof}

On the other hand, as pointed out by Barzdin', the above result fails for $D$.
We prove this in a strong form.

\begin{thm} There is a c.e.\ set $A$ such that $D(A,B) = 1$ for every
  co-c.e.\ set $B$.
\end{thm}

\begin{proof} Let $I_n$ denote the interval $[n!, (n+1)!)$.  Define $A
  = \cup_n (W_n \cap I_n)$.  Fix $n$, and let $S = \{e : W_e = W_n
  \}$.  Then $A \triangle \overline{W_n} \supseteq \cup_{e \in S}
  I_e$.  The latter set has upper density $1$ since $S$ is infinite,
  so $D(A, \overline{W_n}) = 1$.
\end{proof}

Also it is easy to see
that Barzdin's result does not hold for $\epsilon = 0$.

\begin{thm}(\cite{JS}) \label{oldresult} There is a c.e.\ set $A$ such
  that $d(A,B) > 0$ for every co-c.e.\ set $B$.
\end{thm}

This follows at once from the proof of Theorem 2.16 of \cite{JS}.  We
will extend it below in Theorem \ref{nonapprox} by adding the
requirement that the density of $A$ exists.

In \cite{JS}, it was pointed out just after the proof of Theorem 2.21
that every c.e.\ set of upper density $1$ has a computable subset of
upper density $1$.   We now extend this result using the same method of 
proof as in Theorem \ref{Barzdin}.

\begin{thm} \label{d2approx} Let $A$ be a c.e.\ set such that
  $\overline{\rho}(A)$ is a $\Delta^0_2$
  real.  Then $A$ has a computable subset $B$ such that
  $\overline{\rho}(B) = \overline{\rho}(A)$, and hence, by Lemma \ref{diff},
  $d(A,B) = 0$.
\end{thm}
\begin{proof} Let $\{q_s\}_{s \in \omega}$ be a computable sequence of
  rational numbers converging to $\overline{\rho}(A)$.  Define a sequence of
  pairs of natural numbers $(s_n, t_n)_{n \in \omega}$ recursively as
  follows.  Let $(s_0, t_0) = (0,0)$.  Given $(s_n, t_n)$, let
  $(s_{n+1}, t_{n+1})$ be the first pair $(s,t)$ such that:
$$s >  s_n \quad \& \quad t > n \quad \& \quad
\rho_s(A_t \setminus [0,s_n)) \geq q_t - 2^{-n}$$ 
We claim that such a
pair $(s,t)$ exists.  First, choose $s > s_n$ such that $\rho_s(A
\setminus [0,s_n)) \geq \overline{\rho}(A) - 2^{-(n+1)}$.  There are
infinitely many $s$ which satisfy this inequality since
$\overline{\rho}(A) = \overline{\rho}(A \setminus [0,s_t))$.  Now
choose $t > n$ such that $\rho_s(A \setminus [0,s_n)) = \rho_s(A_t \setminus [0,s_n))$ and $q_t \leq
\overline{\rho}(A) + 2^{-(n+1)}$.  Any sufficiently large $t$ meets
these conditions since $\lim_t q_t = \overline{\rho}(A)$.  Then
$$\rho_s(A_t \setminus [0,s_n)) = \rho_s(A \setminus [0,s_n)) \geq \overline{\rho}(A) - 2^{-(n+1)} \geq
q_t - 2^{-(n+1)} - 2^{-(n+1)} = q_t - 2^{-n}$$ Hence the chosen pair
$(s,t)$ meets the condition above to be chosen as $(s_{n+1},t_{n+1})$.
It is easy to see that the sequence $(s_n, t_n)$ is computable.  Let
$S_n$ be the interval $[s_n, s_{n+1})$, so that every natural number
belongs to $S_n$ for exactly one $n$.

We now define the desired computable $B \subseteq A$.  For $k \in
S_n$, put $k$ into $B$ if and only if $k \in A_{t_{n+1}}$.  Clearly,
$B$ is a computable subset of $A$.  Hence $\overline{\rho}(B) \leq
\overline{\rho}(A)$.  To get the opposite inequality, note that $B$
and $A_{t_{n+1}}$ agree on the interval $S_n$.  Further, by definition
of $(s_{n+1}, t_{n+1})$, we have $\rho_{s_{n+1}}(A_{t_{n+1}} \setminus
[0,s_n)) \geq q_{t_{n+1}}$.  It follows from the definition of $B$ that
$$\rho_{s_{n+1}}(B) \geq q_{t_{n+1}}$$
for all $n$.   Therefore:
$$\overline{\rho}(B) \geq \limsup_n \rho_{s_{n+1}}(B) \geq
\limsup_n q_{t_{n+1}} = \overline{\rho}(A)$$
as needed to complete the proof. 
\end{proof}

We now show that we cannot omit the hypothesis that $\overline{\rho}(A)$ is a
$\Delta^0_2$ real from the above theorem, even if we assume in addition that
$\rho(A)$ exists.

\begin{thm} \label{nonapprox} There is a c.e.\ set $A$ such that
  density of $A$ exists, yet for every $\Pi^0_1$ subset $B$ of $A$, we
have $\underline{\rho}(A \setminus B) > 0$ and hence, by Lemma \ref{diff},
$d(A,B) > 0$.
\end{thm}

\begin{proof}
  Recall that $R_e = \{x : 2^e \mid x \ \& \ 2^{e+1} \nmid x\}$.  For
  $x \in R_e$, put $x$ into $A$ if and only if every $y \leq x$ with
  $y \in R_e$ is in $W_e$.  We first show that, for each $e$, if
  $\overline{W_e} \subseteq A$, then $\overline{\rho}(\overline{W_e})
  < \overline{\rho}(A)$, and then show $A$ has a density.  Let $e$ be given.

  {\bf Case 1}.  $R_e \subseteq W_e$ .  Then $R_e \subseteq A$ by
  definition of $A$.  So $R_e \subseteq (A \setminus \overline{W_e})$, and
  hence $A \setminus \overline{W_e}$ has positive lower density.

{\bf Case 2}.  Otherwise.  Take $x \in R_e \setminus W_e$.  Then $x
\notin A$ by definition of $A$ , so $x \notin A \cup W_e$.  Hence
$\overline{W_e}$ is not a subset of $A$.  Note further in this case
that $R_e \cap A$ is finite.

To see that $A$ has a density, note by the above that, for all e,
either $R_e \subseteq A$ or $R_e \cap A$ is finite, so that $R_e \cap
A$ has a density for all $e$.  It follows by restricted countable
additivity (Lemma 2.6 of \cite{JS}) that $A$ has a density, namely
$$\rho(A) = \sum_e \rho(A \cap R_e) = \sum_e \{2^{-(e+1)} : R_e \subseteq
W_e\}$$
\end{proof}

We now look at analogues of some of the above results where we
study $D(A,B)$ instead of $d(A,B)$, for $B$ a computable subset
of a given c.e.\ set $A$.   It was shown in Theorem \ref{Barzdin}
that for every c.e.\ set $A$ and real number $\epsilon > 0$ there is
a computable set $B \subseteq A$ with $d(A,B) < \epsilon$.   We pointed out
in Theorem \ref{oldresult} that the corresponding result fails for $D$ in place
of $d$, but we now show that this corresponding result does hold if we assume
$A$ has a density.

\begin{thm} \label{approx} Let $A$ be a c.e.\ set and $\epsilon$ a positive real
  number.  Then $A$ has a computable subset $B$ such that
  $\underline{\rho}(B) > \underline{\rho}(A) - \epsilon$.
 \end{thm}

 We give a corollary before proving this result.  Roughly speaking,
 this corollary asserts that the computable sets are topologically
 dense among the c.e.\ sets which have an asymptotic density
 (pretending that the pseudometric $D$ is a metric). 

 \begin{cor} \label{eps} Let $A$ be a c.e.\ set which has a density and
   $\epsilon$ a positive real number.  Then $A$ has a computable
   subset $B$ such that $D(A,B) < \epsilon$.
\end{cor}

\begin{proof} (of corollary).  By the theorem, let $B$ be a computable
  subset of $A$ such that $\underline{\rho}(B) > \underline{\rho}(A) -
  \epsilon$.   By Lemma \ref{diff}, 
$$D(A,B) \leq \overline{\rho}(A) - \underline{\rho}(B) = \underline{\rho}(A) 
- \underline{\rho}(B) < \epsilon$$ 
\end{proof}

\begin{proof} (of theorem) Let $A$ be a c.e.\ set and let $\epsilon$ be
  a positive real number.  We must construct a computable set $B
  \subseteq A$ such that $\underline{\rho}(B) > \underline{\rho}(A) -
  \epsilon$.  Let $q$ be a rational number such that
  $\underline{\rho}(A) - \epsilon < q < \underline{\rho}(A)$.  Since
  $q < \underline{\rho}(A)$ there is a number $n_0$ such that
  $\rho_n(A) \geq q$ for all $n \geq n_0$.  Given $n \geq n_0$, let
  $s(n)$ be the least number $s$ such that $\rho_n (A_s) \geq q$,
  where such an $s$ exists because $\rho_n(A) \geq q$.  Then,
  for each $k \geq  \sqrt{n_0}$, define
$$t(k) = \max \{s(n) :  n_0 \leq n \leq k^2\}$$
Finally, define
$$B = \{k :  k \in A_{t(k)}\}$$
The set $B$ is computable because the functions $s$ and $t$ are
computable.  Note that in deciding whether to put $k$ into $B$, we are
waiting for sufficient elements to be enumerated in $A$ on the
interval $[0, k^2)$, which for large $k$ is much bigger than the
interval $[0,k)$.  Such a ``look-ahead'' is crucial to our argument.

Suppose now that $k \geq \sqrt{n}$, and $n \geq n_0$.  Then $n \leq
k^2$, so $s(n) \leq t(k)$, and hence $A_{s(n)} \subseteq A_{t(k)}$.
Thus, for $n \geq n_0$, every number $k \in A_{s(n)}$ with $k \geq
\sqrt{n}$ is in $B$, by the definition of $B$.  It follows that
$$|B \cap [0,n)| \ \geq \  |A_{s(n)}| - \sqrt{n}$$
Since $\rho_n (A_{s(n)}) \geq q$, division by $n$ yields that
$$ \rho_n(B) \geq q - 1/\sqrt{n}$$
for $n \geq n_0$.  As $n$ approaches infinity, $1/\sqrt{n}$ tends to $0$, and
hence $\underline{\rho}(B) \geq q > \underline{\rho}(A) - \epsilon$.
\end{proof}

If $A$ is a c.e.\ set of density $1$, it would be tempting to try to
show that $A$ has a computable subset $B$ of density $1$ by using the
method of the previous theorem applied to values of $q$ closer and
closer to $1$.  However, this breaks down because $n_0$ need not
depend effectively on $q$, so we do not have an effective way to
handle the finitely many ``bad'' $n < n_0$ as $q$ varies.  Indeed,
this breakdown is essential, as it is shown in \cite{JS}, Theorem
2.22, there is a c.e.\ set of density $1$ with no computable subset of
density $1$.  On the other hand, if we \emph{assume} that $A$ is such
that $n_0$ depends effectively on $q$, this plan goes through.  We
make this explicit in the following definition and theorem.

\begin{defn} Let $A$ be a set of density $1$.
    \begin{itemize}

    \item[(i)] A function $w$ \emph{witnesses} that $A$ has density $1$ if
      $(\forall k)(\forall n \geq w(k))[\rho_n(A) \geq 1 - 2^{-k}]$.

      \item[(ii)]  The set $A$ has density $1$ \emph{effectively} if there is
       a computable function $w$ which witnesses that $A$ has density $1$.
    \end{itemize}
\end{defn}

\begin{thm} \label{eff} If $A$ is c.e.\ and has density $1$
  effectively, then $A$ has a computable subset $B$ which has density
  $1$ effectively.
\end{thm}

\begin{proof} Let $w$ be a computable function which witnesses that
  $A$ has density $1$ and let $\{A_s\}$ be a computable enumeration of
  $A$. For $n \geq 0$ let $s(n)$ be the least $s$ such that
  $\rho_n(A_s) \geq 1 - 2^{-z}$ for all $z \leq n$ such that $w(z)
  \leq n$.  The function $s$ is total because $w$ witnesses that $A$
  has density $1$.  We now define the
  function $t$ and the set $B$ exactly as in the previous theorem,
  namely
$$t(k) = \max \{s(n) : n \leq k^2\}$$
$$B = \{k : k \in A_{t(k)}\}$$
As before, $B$ is computable because the functions $s$ and $t$ are
computable, and clearly $B \subseteq A$.  Further, we can argue
exactly as in the previous theorem that if $w(z) \leq n$ and $z \leq
n$, then
$$\rho_n (B) \geq 1 - 2^{-z} - 1/\sqrt{n}$$
Let $h(n)$ be the greatest number $z \leq n$ with $w(z) \leq n$.  (We
may assume without loss of generality that $w(0) = 0$, so such a $z$
always exists.)  By the inequality above, we have, for $n > 0$,
$$\rho_n(B) \geq 1 - 2^{-h(n)} - 1/\sqrt{n}$$
Since $h(n)$ tends to infinity as $n$ tends to infinity, it follows that $B$
has density $1$.  Let $b(n) = 1 - 2^{-h(n)} - 1/\sqrt{n}$ be the lower
bound for $\rho_n(B)$ obtained above.  Since the function $h$ is
nondecreasing and computable, the function $b$ is also nondecreasing, and
$b(n)$ is a computable real, uniformly in $n$.  Also $\lim_n b(n) =
1$.  It follows that $B$ has density $1$ effectively.
\end{proof}

Note that if $A$ has density $1$, then there is a function $w_A \leq_T
A'$ which witnesses that $A$ has density $1$, namely
$$w_A(k) = (\mu y) (\forall n \geq y) [\rho_n(A) \geq 1 - 2^{-k}]$$
We call $w_A$ the \emph{minimal witness function} for $A$.
In particular, if $A$ is a c.e.\ set of density $1$, then there is a function $w
\leq_T 0''$ which witnesses that $A$ has density $1$.  The next result
shows that if there is such a $w \leq_T 0'$, then $A$ has a computable
subset of $B$ of density $1$.

\begin{thm} \label{delta2w}
Let $A$ be a c.e.\ set of density $1$.  Then the following are equivalent:
\begin{itemize}
      \item[(i)] $A$ has a computable subset $B$ of density $1$.
       \item[(ii)]  There is a function $w \leq_T 0'$ which witnesses that $A$ has
density $1$.
\end{itemize}
\end{thm}

\begin{proof} First, assume that (i) holds.  Then by the remark just
  above the statement of the theorem, there is a function $w_B \leq_T
  0'$ which witnesses that $B$ has density $1$.  Since $A \supseteq
  B$, we have that $\rho_n(A) \geq \rho_n(B)$ for all $n$, and so $w_B$
  also witnesses that $A$ has density $1$.

  Assume now that (ii) holds.  We will now prove (i) using the method
  of Theorem \ref{eff}, but using a computable approximation to $w$ in
  place of $w$.  The basic trick in proving Theorem \ref{eff} was to enumerate
  elements in $A$ until sufficient elements appeared to show that the
  density of $A$ on a given interval is at least as big as the lower
  bound given by $w$.  This would seem to carry the danger now that if
  our approximation to $w$ is incorrect, $A$ may not have sufficient
  elements in the interval to make its density at least as big as
  predicted by the approximation, and we would wait forever, causing
  the construction to bog down.  The solution to this is both simple
  and familiar.  As we wait for the elements to appear in $A$ we
  recompute the approximation.  Since the approximation converges to
  $w$, eventually sufficient elements must appear in $A$ for some
  sufficiently late approximation.

   We now implement the above strategy.  Let $g(.,.)$ be a computable function
such that $(\forall k)[w(k) = \lim_s g(k,s)]$.  Define
$$s(n) = (\mu s \geq n)(\forall k \leq n)[ g(k,s) \leq n \rightarrow
\rho_n (A_s) \geq 1 - 2^{-k}]$$ 
Note that the variable $s$ occurs both
as an argument of $g$ and as a stage of enumeration of $A$, in
accordance with our informal description of the strategy.  The
function $s$ is total because all sufficiently large numbers $s$
satisfy the defining property for $s(n)$, since $w$ witnesses that $A$
has density $1$.  We now define the computable function $t$ and the
computable set $B \subseteq A$ exactly as in Theorems \ref{approx} and
\ref{eff}.  This yields that $B$ is computable, $B \subseteq A$, and
for each $n \geq 0$, $\{k \geq \sqrt{n} : k \in A_{s(n)}\} \subseteq
B$.  These are proved just as in the proof of Theorem \ref{eff}.

We now show that $B$ has density $1$.  Let $b$ be given.  Suppose $n$
is sufficiently large that $n > b$, $n \geq w(b)$, and $(\forall s
\geq n) [g(b,s) = w(b)]$.  Then by definition of $s(n)$ (with $k = b)$,
$\rho_n (A_{s(n)}) \geq 1 - 2^{-b}$.  We then have, as in the proof of
Theorem \ref{delta2w}, for all sufficiently large $n$,
$$\rho_n(B) \geq \rho_n(A_{s(n)}) - 1/\sqrt{n} \geq 1 - 2^{-b} - 1/\sqrt{n}
\geq 1 - 2^{-b} - 1/\sqrt{b}$$
Since $\lim_b (1 - 2^{-b} - 1/\sqrt{b}) = 1$,
it follows that $\rho(B) = \lim_n \rho_n (B) = 1$.
\end{proof}

\begin{cor} \label{low} Suppose that $A$ is a low c.e.\ set of density
  $1$.  Then $A$ has a computable subset of density $1$.
\end{cor}
\begin{proof} As remarked just before the statement of Theorem
  \ref{delta2w}, there is a function $w \leq_T A'$ which witnesses
  that $A$ has density $1$.  Since $A$ is low, we have $w \leq_T 0'$,
  and hence $A$ has a computable subset of density $1$ by Theorem
  \ref{delta2w}.
\end{proof}

In the next section we will see that, conversely, every nonlow c.e.
degree contains a c.e.\ set of density $1$ with no computable subset of
density $1$.

We now use similar ideas to extend Corollary \ref{low} from sets of
  density $1$ to sets whose lower density is a $\Delta^0_2$ real.

  \begin{thm} \label{delta2r} Let $A$ be a low c.e.\ set such that
    $\underline{\rho}(A)$ is a $\Delta^0_2$ real.  Then $A$ has a
    computable subset $B$ such that $\underline{\rho}(B) =
    \underline{\rho}(A)$ and hence, by Lemma \ref{diff}, $d(A,B) = 0$.
\end{thm}

\begin{proof} The proof is similar to that of Theorem \ref{delta2w}.
  Let $\{q_n\}$ be a computable sequence of rational numbers
  converging to $\underline{\rho}(A)$.  Define
$$w(k) = (\mu y) (\forall n \geq y) [\rho_n(A) \geq q_n - 2^{-k}]$$
Observe that $w$ is a total function since for each $k$, whenever $n$
is sufficiently large we have $\rho_n(A) \geq q_n - 2^{-k}$, because
$\{q_n\}$ converges to $\liminf_n \rho_n(A)$.  Note also that
$\rho_n(A)$ is a rational number which can be computed from $n$ and an
oracle for $A$.  Hence $w \leq_T A' \leq_T 0'$, so there is a
computable function $g$ such that, for all $k$, $w(k) = \lim_s
g(k,s)$.  Now define:
$$s(n) = (\mu s \geq n)(\forall k \leq n) [ g(k,s) \leq n \implies 
\rho_n(A_s) \geq q_n - 2^{-k}]$$
$$t(k) = \max \{s(n) : n \leq k^2\}$$
$$B = \{k : k \in A_{t(k)} \}$$
The function $s$ is total because for each $n$ and $k \leq n$ all
sufficiently large numbers $s$ satisfy the matrix of the definition of
$s(n)$.  It follows that the functions $s$ and $t$ and the set $B$ are
computable, and obviously $B \subseteq A$.  It follows from the latter
that $\underline{\rho}(B) \leq \underline{\rho}(A)$, so it remains
only to verify that $\underline{\rho}(B) \geq \underline{\rho}(A)$.
For this, note that, just as in the proof of Theorem \ref{delta2w},
for all $n > 0$
$$\{k \geq \sqrt{n} : k \in A_{s(n)}\} \subseteq B \mbox{ and hence } 
\rho_n(B) \geq \rho_n(A_{s(n)}) - 1/\sqrt{n}$$ 
Now let $b > 0$ be given.  Let $n$ be
sufficiently large that $n > b, n \geq w(b), (\forall s \geq n)[g(b,s)
= w(b)]$, and $|\underline{\rho}(A) - q_n| < 2^{-b}$.  It then follows
that
$$\rho_n (A_{s(n)}) \geq q_n - 2^{-b}$$
by using the above conditions on $n$ and the definition of $w(n)$ with
$k = b$.  We now have:
$$\rho_n(B) \geq \rho_n(A_{s(n)}) - 1/\sqrt{n} \geq q_n - 2^{-b} - 1/\sqrt{n}
\geq \underline{\rho}(A) - 2^{-b} - 1/\sqrt{b}$$
 Hence $\underline{\rho}(B) \geq \underline{\rho}(A) - 2^{-b} - 1/\sqrt{b}$.
Since $b > 0$ was arbitrary and $\lim_b (2^{-b} + 1/\sqrt{b}) = 0$, we have
$\underline{\rho}(B) \geq \underline{\rho}(A)$.
\end{proof}

It was shown in \cite{JS}, Theorem 2.21, that if a computable set $A$
has a density $d$, then $d$ is a $\Delta^0_2$ real.  (Actually, this
part of the theorem is an immediate consequence of the Limit Lemma.)
It follows by relativizing the proof that if a low set $A$ has density
$d$, then $d$ is a $\Delta^0_2$ real.  This gives the following
corollary.

\begin{cor} \label{low1} If $A$ is a low c.e.\ set and $\rho(A)$
  exists, then $A$ has a computable subset $B$ with $\rho(B) =
  \rho(A)$ and hence, by Lemma \ref{diff}, $D(A,B) = 0$.  (Recall that
  $D(A,B)$ is the upper density of the symmetric difference of $A$ and $B$.)
\end{cor}
\begin{proof} As noted just above, $\rho(A)$ is a $\Delta^0_2$ real,
  so by Theorem \ref{delta2r}, $A$ has a computable subset $B$ with
  $\underline{\rho}(B) = \underline{\rho}(A) = \rho(A)$.  Further,
  $\overline{\rho}(B) \leq{\rho(A)}$ since $B \subseteq A$.  Finally,
  $\overline{\rho}(B) \geq \underline{\rho}(B) = \rho(A)$, so
  $\overline{\rho}(B) = \rho(A)$.  As $\underline{\rho}(B) =
  \overline{\rho}(B) = \rho(A)$, we have $\rho(B) = \rho(A)$.
\end{proof}

The next result uses our previous work to characterize the densities
of computable subsets of those low c.e.\ sets $A$ which have a density
$d$.  For $d_0$ to be the density of a computable subset of $A$ it is
clearly necessary that $0 \leq d_0 \leq d$ and (by Theorem 2.21 of
\cite{JS}) that $d_0$ be a $\Delta^0_2$ real.  We now show that these
conditions are also sufficient.

\begin{cor}\label{subset} Let $A$ be a low c.e.\ set of density $d$ and let $d_0$ be
  a $\Delta^0_2$ real such that $0 \leq d_0 \leq d$.  Then $A$ has a
  computable subset $B$ of density $d_0$.
\end{cor}

\begin{proof}  By Corollary \ref{low1}, $A$ has a computable subset
  $A_0$ of density $d$.  Thus, we may assume without loss of
  generality that $A$ is computable, since we can simply replace $A$
  by $A_0$.

  By \cite{JS}, Theorem 2.21, every $\Delta^0_2$ real in $[0,1]$ is
  the density of a computable set.  Using the same proof but working
  within A we get that every $\Delta^0_2$ real $s \in [0,1]$ is the
  relative density within $A$ of a computable subset $B$ of $A$, i.e.
  $\rho(B | A) = s$. The result to be proved is immediate if $d = 0$,
  so assume $d > 0$ and hence $s = \rho(B \mid A) = \rho(B) /
  \rho(A)$.  We now choose $s = d_0 / d$. ($s$ is a $\Delta^0_2$ real
  since the $\Delta^0_2$ reals form a field by relativizing to $0'$
  the result that the computable reals form a field.  Also $0 \leq s_0
  \leq 1$ since $0 \leq d_0 \leq d$.)  Let $B$ be a computable subset
  of $A$ such that $\rho(B \mid A) = \rho(B) / \rho(A) = s = d_0 / d$.
  Multiply both sides by $d = \rho(A)$, to obtain $\rho(B) = d_0$ as
  needed.
\end{proof}

The following theorem greatly strengthens Theorem 2.22 of \cite{JS},
which asserts that there is a c.e.\ set of density $1$ which has no
computable subset of density $1$.  It contrasts strongly with
Corollary \ref{subset}.

\begin{thm}\label{nononzero} There is a c.e.\ set $A$ of density $1$ 
such that no computable subset of $A$ has nonzero density.
\end{thm}

\begin{proof}
  For each $e$, let $S_e = \{n : \varphi_e(n) = 1\}$, so that the
  computable sets are exactly the sets $S_e$ with $\varphi_e$ total.
  Let $N_e$ be the requirement: 
$$N_e : (\varphi_e \mbox{ total } \ \&
\ S_e \subseteq A \ \& \ \rho(S_e) \downarrow) \ \Longrightarrow \
\rho(S_e) = 0$$ To prove the theorem, it suffices to construct a c.e.\
set $A$ of density $1$ which meets all the requirements $N_e$.  The
strategy for meeting $N_e$ is as follows.  We define a sequence of
finite intervals $I_{e,0}, I_{e,1}, \dots$, and this sequence may or
may not terminate, and the strategy affects $A$ only on these
intervals.  These intervals are pairwise disjoint and also disjoint
from all intervals used for other requirements, so distinct
requirements $N_e$ don't interact.  The intervals are defined in the
order listed above.  When $I_{e,j}$ is chosen, its least element
$a_{e,j}$ should be the least number not in any interval already
chosen for any requirement.  (The purpose of this is to ensure that
every number belongs to some interval for some requirement.)  Further,
we will carefully choose a certain large initial segment $J_{e,j}$ of
$I_{e,j}$, but we defer the definition of $J_{e,j}$ for the moment.
As soon as $I_{e,j}$ (and hence $J_{e,j}$) are chosen, put all elements
of $J_{e,j}$ into $A$.  (This is done to help ensure that $A$ has
density $1$.)  Then wait for a stage $s_{e,j}$ at which $\varphi_e$ is
defined on all elements of $I_{e,j}$.  (If this never occurs, it
follows that $\varphi_e$ is not total and hence $N_e$ is met
vacuously.)  If $\varphi_e(x) = 1$ for some $x \in I_{e,j} \setminus
J_{e,j}$, we let $I_{e,j}$ be the final interval for $N_e$ and take no
further action for $N_e$.  In this case, $N_e$ is met because $S_e$ is
not a subset of $A$, as $x \in S_e \setminus A$, for the $x$ just
mentioned.  If there is no such $x$, put all elements of $I_{e,j}
\setminus J_{e,j}$ into $A$ at stage $s_{e,j} + 1$, thus ensuring
$I_{e,j} \subseteq A$.  Then define $I_{e, j+1}$ as above at the next stage
devoted to $N_e$.

The idea of the above strategy is that if we define $I_{e,j+1}$ and
$S_e$ has density $d$, then the density of $S_e$ up to $\max J_{e,j}$
should be approximately $d$, while the density of $S_e$ on the
interval $(\max J_{e,j}, \max I_{e,j}]$ is surely $0$, as $S_e$ does
not intersect this interval.  If the latter interval is large, this
suggests that $d$ is close to $0$, and in fact we get $d=0$ by taking
a limit.  Of course, we also must make $|J_{e,j}|$ a large fraction of
$|I_{e,j}|$ to ensure that $A$ has density $1$.  Although these two
largeness requirements go in opposite directions, it is easy to meet
both of them, as the following calculations show.

Holding $e,j$ fixed for now, let $a = \min I_{e,j}$, $b = \max
J_{e,j}$, and $c = \max I_{e,j}$.  Note that $a \leq b \leq c$ because
$J_{e,j}$ is an initial segment of $I_{e,j}$.  We have already
determined $a$ as the least number not in any previously defined
interval.  In order to meet $N_e$, we make the ratio $b/c$ strictly
less than $1$ and independent of $k$.  Specifically, we require that
$b/c = 1 - 2^{-(e+1)}$.  In order to ensure that $A$ has density $1$
we also wish $|J_{e,j}| / |I_{e,j}| = \frac{b-a+1}{c-a+1}$ to have a
lower bound which depends only on $e$ and approaches $1$ as $e$
approaches infinity.  But, for fixed $a$, if $b$ approaches infinity and
$b$ and $c$ are large and
related as above, then $\frac{b-a+1}{c-a+1}$ approaches
$b/c$, which equals $1 - 2^{-(e+1)}$.    Thus, we may choose $b$ sufficiently
large that $\frac{b-a+1}{c-a+1} \geq 1 - 2^{-e}$, and of course this determines
$c$, so the intervals $I_{e,j}, J_{e,j}$ are determined.

We claim that the above strategy suffices to satisfy $N_e$.  This is
obvious if there are only finitely many intervals $I_{e,j}$, since in
this case either $\varphi_e$ is not total or $S_e \nsubseteq A$, and
$N_e$ is satisfied vacuously.  Suppose now there are infinitely many
such intervals, so that $I_{e,j}$ is defined for every $j$.  Note that
$S_e \cap I_{e,j} \subseteq J_{e,j}$ for all $j$.  For the moment, let
$e$, $j$ be fixed and drop the subscript $(e,j)$ from $a, b$, and $c$.
We now calculate the decrease in density of $S_e$ as we go from $b$ to
$c$ without seeing any elements of $S_e$.  Let $r = |S_e \cap [0,b]$,
so $\rho_b (S_e) = r / b$.  Then:
$$\rho_b(S_e) - \rho_c(S_e) = \frac{r}{b} - \frac{r}{c} = \frac{r}{b}(1 - 
\frac{b}{c}) = \rho_b (S_e)  2^{-(e+1)}$$
Assume now that $\rho(S_e)$ exists, since otherwise $N_e$ is vacuously met.
Letting the (unwritten) $j$ in the above equation tend to infinity yields:
$$\rho(S_e) - \rho(S_e) = \rho(S_e)2^{-(e+1)}$$
It follows that $\rho(S_e) = 0$, and so $N_e$ is met.

It remains to show that $A$ has density $1$.  Let $E$ be the set of
all points of the form $\max I + 1$, where $I$ is any interval used in the
construction.  We first show that $\lim_{c \in E} \rho_c(A) = 1$.
Since every element of $\omega$ belongs to one and only one interval
used in the construction, we see that, for $c \in E$, $\rho_c(A)$ is
the weighted average of the density of $A$ for each interval $I$ used
in the construction with $\max I <  c$, where $I$ has weight $|I|$.
(Here the density of $A$ on $I$ is $|A \cap I| / |I|$.)  If $I$ is
used for the sake of $N_e$ (i.e. $I = I_{e,j}$ for some $j$), by
construction the density of $A$ on $I$ is either equal to $1$ or is at
least $1 - 2^{-e}$, where for each $e$, there is a most one $j$ with
this density not equal to $1$ (i.e. the greatest $j$ such that
$I_{e,j}$ exists).  Thus, for each real $q < 1$, $A$ has density at
least $q$ on all but finitely many intervals used in the construction.
Given $q < 1$, let $b \in E$ be sufficiently large that $A$ has
density at least $q$ on every interval $I$ used in the construction
with $\min I \geq b$.  If $c \in E$ and $c > b$, then $\rho_c(A)$ is
the weighted average of the density of $A$ on $[0,b)$ and the density
of $A$ on $[b,c)$, where the weight of each interval is its size.  The
latter density is at least $q$, and its weight approaches infinity as
$c$ goes to infinity, while the weight of the former density stays
fixed.  It follows that $\liminf_{c \in E} \rho_c(A) \geq q$.  As $q < 1$ was
arbitrary, it follows that $\lim_{c \in E} \rho_c(A) = 1$.

We now complete the proof that $A$ has density $1$. Let $I$ be any
interval used in the construction, and let $J = I \cap A$.  Let $I =
[a,c]$.  By construction, $J$ is an initial segment of $I$, so as we
examine $\rho_b(A)$ for $b - 1 \in I$, we note that this density
increases until we reach $\max J + 1$ and then decreases until we
reach $c+1$.  It follows that for every $b$ with $b-1 \in I$, either
$\rho_b(A) \geq \rho_{a+1}(A)$ or $\rho_b(A) \geq \rho_{c+1}(A)$.
Furthermore, $a \in A$, so $\rho_{a+1}(A) \geq \rho_{a}(A)$ and $a, c+1
\in E$.  As $b$ goes to infinity, the points $a, c+1$ also go to
infinity, and so $\rho_{a}(A), \rho_{c+1}(A)$ each approach $1$, since
$\lim_{c \in E} \rho_c(A) = 1$.  Since $\rho_b(A) \geq \min
\{\rho_a(A), \rho_{c+1}(A)\}$, it follows that $\rho(A) = \lim_b
\rho_b(A) = 1$. 
\end{proof}

\section{Turing degrees, density, and the outer splitting property}

It was shown in \cite{JS}, Theorem 2.22, that there is a c.e.\ set of
density $1$ which has no computable subset of density $1$.   In this
section we study the degrees of such sets and of their subsets of
density $1$. We also apply the techniques developed for this problem to
study the degrees of sets with properties arising in the study of the
lattice of c.e.\ sets.

\begin{thm} \label{high}   There is a c.e.\ set $A$ such that $A$ has density $1$
and every set $B \subseteq A$ of density $1$ is high, i.e. $B' \geq_T 0''$.
\end{thm}

\begin{proof}

 Recall that $R_e = \{x : 2^e \mid x \ \& \ 2^{e+1} \nmid x\}$.
 As shown in the proof of Theorem 2.22 of \cite{JS},
to ensure that $A$ has density $1$, it suffices to meet the following
positive requirements:
$$P_n : R_n \subseteq^* A$$

To ensure that every subset of $A$ of density $1$ is high, we make
the minimal witness function $w_A$ for $A$ grow very fast.  Specifically,
define 
$$w_A(n) = (\mu b) (\forall k \geq b) [\rho_k(A) \geq 1 - 2^{-n}]$$
In order to ensure that every subset $B$ of $A$ of density $1$ is 
high, it suffices to meet the following negative requirements:
$$N_{n} : |W_n| < \infty \quad  \Longrightarrow \quad w_A(n+2) \geq \max (W_n \cup \{0\})$$
To see that it suffices to meet the given requirements, assume that
$A$ satisfies all the positive and negative requirements.  Let $B$ be
a subset of $A$ of density $1$, and let $w_B$ be the corresponding
minimal witness function for $B$, defined as above with $A$ replaced by $B$.
Clearly, $w_B(n) \geq w_A(n)$ for all $n$, since $B \subseteq A$, and
so each requirement $N_n$ holds with $A$ replaced by $B$.  Also, $w_B$
is total because $B$ has density $1$, and $w_B \leq_T B'$.   Let
Inf $= \{n : |W_n| = \infty\}$.  Then for all $n$,
$$n \in \mbox{Inf }  \Longleftrightarrow (\exists x \in W_n)[ x > w_B(n+2)]$$
It follows that $$0'' \leq_T \mbox{Inf} \leq_T w_B \oplus 0' \leq_T
B'$$ since $B'$ can calculate $w_B(n+2)$ and then $0'$ can determine
whether $W_n$ has an element exceeding $w_B(n+2)$.  It follows that
$B$ is high, as needed.

The strategy for meeting $P_n$ is, at each stage $s$, to enumerate
each $x \in R_n$ with $x \leq s$ into $A$ unless $x$ is restrained by
$N_n$ at the end of stage $s$, as described below.  This will succeed
in meeting $P_n$ because there will be only finitely many numbers
permanently restrained by $N_n$.

We now give the strategy for meeting the requirement $N_n$, where this
strategy is similar to that used in Theorem 2.22 of \cite{JS}.  This
strategy restrains $A$ only on $R_n$ and so interacts only with the
requirement $P_n$.  Say that a finite nonempty set $I \subseteq R_n$
is $n$-\emph{large} if $\rho_m(I) > 2^{-(n+2)}$, where $m = \max I$.
Since $\rho(R_n) > 2^{-(n+2)}$, for each $a$, the set $[a,b] \cap R_n$
is $n$-large for all sufficiently large $b$.  Also, if $I \subseteq
R_n$ is $n$-large and disjoint from $A$, we have $\rho_m (A) \leq 1 -
\rho_m(I) < 1 - 2^{-(n+2)}$, where $m = \max I$.  It follows in this
case that $w_A(n+2) \geq m$.  Thus to meet $N_n$, it suffices to
ensure that, if $W_n$ is finite, there is an $n$-large set $I$ which
is disjoint from $A$ with $\max I > \max(W_n \cup \{0\})$.  To achieve
this, start with any $n$-large set $I_0 \subseteq R_n$ currently
disjoint from $A$ and with $\max I_0$ exceeding all elements currently
in $W_n \cup \{0\}$.  Restrain all elements of $I_0$ from entering $A$
until, if ever, a stage $s_0$ is reached at which a number exceeding
$\max I_0$ is enumerated in $W_n$.  At stage $s_0$, enumerate all
elements of $I_0$ into $A$ (for the sake of $P_n$), and start over
with a new interval $I_1$ which is $n$-large and currently disjoint
from $A$ and satisfies $\max(I_1) > \max(I_0)$.  Proceed in the same
way, restraining all elements of $I_1$ from $A$ until, if ever, $W_e$
enumerates an element greater than $\max(I_1)$, in which case you
proceed to $I_2$, etc.  Now if $I_k$ exists for every $k$, then $W_n$
is infinite, since $\max(I_0) < \max(I_1) < \dots$ and, for each $k$,
$W_n$ contains an element exceeding $\max(I_k)$.  Thus $N_n$ is met
vacuously in this case.  Also, $P_n$ is met because $R_n \subseteq A$,
as infinitely often all restraints are dropped.  Otherwise, there is a
largest $k$ such that $I_k$ exists.  Then, for this $k$, $I_k$ is the
desired $n$-large set disjoint from $A$ with $\max(I_k) > \max(W_n
\cup \{0\}$, so $N_n$ is met.   The requirement $P_n$ is met because
$R_n \setminus I_k \subseteq A$.
\end{proof}

Eric Astor (private communication) has observed that every c.e.\ set
of density $1$ has subsets of density $1$ in every high degree.   This
allows us to strengthen the theorem as follows:

\begin{cor} (with Astor) \label{exhigh} There is a c.e.\ set $A$ of
  density $1$ such that the degrees of the subsets of $A$ which have
  density $1$ are precisely the high degrees.
\end{cor}

\begin{proof}
  Let $A$ be any c.e.\ set of density $1$ such that every subset of $A$
  of density $1$ is high.  To complete the proof, it suffices to show
  that, for every set $B$ of high degree, $A$ has a subset $C$ of
  density $1$ which is Turing equivalent to $B$.  Let $w_A$ be the
  minimal witness function for $A$ as defined just before the
  statement of Theorem \ref{delta2w}, and suppose that $B$ has high
  degree.  Note that $w_A \leq_T A' \leq_T 0'' \leq_T B'$.  By
  relativizing the proof of Theorem \ref{delta2w} to $B$, we see that $A$
  has a subset $C_0 \leq_T B$ such that $\rho(C_0) = 1$.  We now use a
  simple coding argument so obtain a set $C \subseteq A$ which has
  density $1$ and is Turing equivalent to $B$.  Let $R$ be an infinite
  computable subset of $A$ which has density $0$.  (To obtain $R$,
  first choose an infinite computable subset $R_0$ of $A$, and then
  show that $R_0$ has an infinite computable subset of density $0$.)
  Then let $C_1$ be a subset of $R$ which is Turing equivalent to
  $B$. Finally, let $C = (C_0 \setminus R) \cup C_1$.  Then $C_1
  \subseteq R \subseteq A$, so $C \subseteq C_0 \cup C_1 \subseteq A$.
  Also, $C$ has density $1$ because $C_0 \setminus R$ has density $1$.
  Further, $C \leq_T B$, because $C_0 \leq_T B$ and $C_1 \leq_T B$.
  Finally, $B \leq_T C_1 \leq_T C$, where $C_1 \leq_T C = (C_0
  \setminus R) \cup C_1$ because $C_0 \setminus R$ and $C_1$ are
  separated by the computable set $R$.  Thus, $C$ is the desired
  subset of $A$ which has density $1$ and is Turing equivalent to $B$.
\end{proof}

Recall that it was shown in Corollary \ref{low} that every low c.e.\ set
of density $1$ has a computable subset of density $1$.  We now show
that, conversely, every nonlow c.e.\ degree computes a c.e.\ set of
density $1$ with no computable subset of density $1$.  This result
extends Theorem 2.22 of \cite{JS}, which asserts the existence of a
c.e.\ set of density $1$ with no computable subset of density $1$, and
gives an example (the first?) of a simple, natural property $P$ of
c.e.\ sets such that the degrees containing c.e.\ sets with the property
$P$ are exactly the nonlow c.e.\ degrees.  We use a similar technique to
show that every nonlow c.e.\ degree contains a c.e.\ set which is not
semilow$_{1.5}$, and use this to show that every such degree contains
a set without the outer splitting property, answering a question raised
by Peter Cholak.

\begin{thm} \label{nonlow} If ${\bf a}$ is any non-low c.e.\ degree
  then it contains a c.e.\ set A of density 1 with no computable
  subset of density 1.
\end{thm}

\begin{proof}
 The existence of a c.e.\ set $A$ of density $1$ with no
  computable subset of density $1$ was proved in \cite{JS}, Theorem
  2.22, and our proof here uses a similar strategy, but with
  permitting added in.  Familiarity with the proof of \cite{JS},
  Theorem 2.22, would be helpful to the reader.

  Given a c.e.\ set $C$ of nonlow degree ${\bf a}$, we construct a
  c.e.\ set $A \leq_T C$ which has density $1$ but has no computable
  (or even co-c.e.) subset of density $1$.  This suffices to prove the
  theorem, since we can then define
$$\hat A = (A \setminus \{2^n : n \in \omega\}) \cup \{2^n : n \in C\}$$
and show that $\hat A$ is a c.e.\ set of degree $\bf a$ which has density
$1$ but has no computable subset of density $1$.

Recall that
$$R_k = \{m : 2^k \mid m \ \& \ 2^{k+1} \nmid m\}$$
  As shown in the proof of Theorem 2.22 of \cite{JS},
to ensure that $A$ has density $1$, it suffices to meet the following
positive requirements:
$$P_n : R_n \subseteq^* A$$

To help us meet these positive requirements, as stage $s$ we put $s$
into $A$ unless it is restrained for the sake of some negative
requirement as described below.  Thus, it is clear that $P_n$ will be
met if the restraint associated with $R_n$ comes to a limit.  As in
the proof of Theorem 2.22 of \cite{JS},  we will show
that $R_n \subseteq A$ if the restraint associated with $R_n$ does not
come to a limit, so that $P_n$ is met in either case.

We make $A \le_T C$ by a slight modification of simple permitting.
Namely, if $x$ enters $A$ at stage $s$, we require that either some
number $y \leq x$ enters $C$ at $s$, or $x = s$.  This obviously
implies that $A \leq_T C$.

As before, let $N_e$ be the statement:
$$N_e :  W_e \cup A = \omega \Rightarrow \overline{\rho}(W_e) > 0$$

The conjunction of the $N_e$'s asserts that $A$ has no co-c.e.\ subset
of density $1$.  Rather than meet the $N_e$'s directly, we split up
each $N_e$ into weaker statements $N_{e,i}$ which will be our actual
requirements.

To do this we will define a computable function $g(e,i,s)$ which ``threatens''
to be a computable approximation to $C'$.  Let $L_{e,i}$ be the
statement:
$$\lim_s g(e,i,s) = C'(i)$$
Then define the {\bf requirement} $N_{e,i}$ as follows:
$$N_{e,i} : N_e \mbox{ or } L_{e,i}$$

Suppose all requirements $N_{e,i}$ are met.   If $N_e$ is not met, then
all $L_{e,i}$ hold and $C$ is low, a contradiction.
Hence, to meet $N_e$ it suffices to meet $N_{e,i}$ for all $i$.

We will meet $N_{e,i}$ by restraining certain elements of $R_{e,i}$
from entering $A$.  We do this in such a way that either the restraint
comes to a limit, or infinitely often all restraint is dropped.

The strategy to meet $N_{e,i}$ is as follows.  We fix $e,i$ and refer
to sets $I$ of the form $[a,b] \cap R_{e,i}$ as \emph{intervals}.  An
interval $I$ is called \emph{large} if at least half of the elements
of $R_{e,i}$ less than $\max I$ are in $I$.  Since $R_{e,i}$ has
positive density, any set which contains infinitely many large
intervals has positive upper density.  At the beginning of each stage
$s$, we have at most one interval, denoted $I[s]$, which is active for
the strategy.  The idea of the strategy is that we set $g(e,i,s) = 0$
while $i \notin C'[s]$, thus threatening to satisfy $L_{e,i}$ via
$C'(i) = 0 = \lim_sg(e,i,s)$ unless $i$ enters $C'$.  If $i$ enters
$C'$, we choose our first interval $I$.  We require that $\min I$
exceed the use of the computation showing $i \in C'$, so that if $i$
leaves $C'$, the elements of $I$ are permitted to enter $A$.  We
choose $I$ so that it does not contain elements already in $A$, and we
restrain elements of $I$ from entering $A$. Thus, if $W_e \cup A =
\omega$, $W_e$ must eventually cover $I$.  While we are waiting for
$W_e$ to cover $I$, we keep $g(e,i,s) = 0$, but when $W_e$ covers $I$,
we change $g(e,i,s)$ to $1$, thus threatening to meet $L_{e,i}$ via
$C'(i) = 1 = \lim_sg(e,i,s)$ unless $i$ leaves $C'$.  If $i$ leaves
$C'$, we dump all elements of $I$ into $A$ (which is permitted because
$C$ changed below $\min I$) and start over, again setting $g(e,i,s) =
0$ and waiting for $i$ to re-enter $C'$, so we can choose a new
interval, etc.  Note that we start over in this fashion whenever $i$
leaves $C'$, whether or not $W_e$ has covered our interval.  If $W_e$
has not covered the interval when we cancel it, we have made progress
on satisfying $L_{e,i}$ via $C'(i) = 0 = \lim_sg(e,i,s)$ because
$C'(i)$ has changed and we have kept $g(e,i,s) = 0$.  If $W_e$ has
covered our interval when we cancel it, then we have made progress on
showing that $W_e$ has positive lower density because $W_e$ contains a
new large interval.  The formal construction and verification are
given below.

Stage $s$.  If $I[s]$ is not defined and $i \in C'[s]$, choose a large
interval $I \subseteq R_{e,i}$ with $I \cap A_s = \emptyset$ and
$\min(I)$ larger than the use of the computation showing $i \in
C'[s]$. Let $I[s+1] = I$.  Let $u_I$ be the use of the computation
showing $i \in C'[s]$ and associate it with $I$ until, if ever it is
cancelled.  Restrain all elements of $I$ from entering $A$ until, if
ever, the interval $I$ is cancelled.

If $I[s]$ is defined and $C_{s+1} - C_s$ contains an element $y \leq
u_I$, then cancel $I[s]$, and enumerate all elements of $I[s]$ into
$A$.  Note that we do this whether or not $W_{e,s} \supseteq I[s]$,
but if $W_{e,s} \supseteq I[s]$, we designate $I[s]$ as a
\emph{successful} interval.  Of course, this enumeration is consistent
with our permitting condition since $u_I \leq \min(I)$.

If neither of the above cases apply, we maintain the current interval and
restraints, if any.   

Finally, in any case define $g(e,i,s)$ to be $1$ if $I[s]$ is defined
and $I[s] \subseteq W_{e,s}$, and otherwise let $g(e,i,s) = 0$.
Furthermore, if $s \in R_{e,i}$ and $s$ is not restrained at the end
of stage $s$, (i.e. $I[s+1]$ is undefined or $s \notin I[s+1]$),
enumerate $s$ into $A$, in addition to any enumeration required above.
This is done to help meet the positive requirement $P_{e,i}$ and is
allowed by our modified permitting condition.  This completes the
description of the construction.

To verify that the construction succeeds in meeting $N_{e,i}$, we consider four cases.

{\bf Case 1}.   For all sufficiently large $s$,  $I[s]$ is undefined.   Then, for all sufficiently
large $s$,  $i \notin C'[s]$ and $g(e,i,s) = 0$.   It follows that $\lim_s g(e,i,s) = 0 = C'(i)$
so that $L_{e,i}$ holds and hence $N_{e,i}$ is met. 

{\bf Case 2}.   There is an interval $I$  with $I[s] = I$ for all sufficiently large $s$.    Then
$i \in C'$ via the same computation as when $I$ was first chosen, since otherwise $I$ would
have been cancelled.    If $W_e \supseteq I$, we have $g(e,i,s) = 1$ for all sufficiently
large $s$.   In this case,  $\lim_s g(e,i,s) = 1 = C'(i)$ and hence $L_{e,i}$ holds.  If
$W_e \not \supseteq I$, then $W_e \cup A \not \supseteq I$, since $I$ is disjoint from $A$, by the
way it was chosen and the restraint imposed.   It follows that $W_e \cup A \neq \omega$, and thus
$N_e$ is met. 

{\bf Case 3}.   There are infinitely many successful intervals $I$.   Then $W_e$ contains all of them and
so has positive upper density.   It follows that $N_e$ is met.  

{\bf Case 4}.   None of Cases 1-3 apply.   In other words, there are infinitely many intervals, but only 
finitely many of them are successful.   Then $i \notin  C'$, since infinitely often the computation
showing $i \in C'$ is destroyed.   Note that $g(e,i,s) = 1$ only if the interval $I[s]$ is successful
or is the final interval.   By the failure of Cases 1-3, there are only finitely many such $s$.
Hence $\lim_s g(e,i,s) = 0 = C'(i)$, and $L_{e,i}$ holds.

We now show that the construction also meets $P_{e,i}$.  Let $U$ be
the union of all intervals ever chosen for $R_{e,i}$.  If $s \in
R_{e,i} - U$, then $s \in A$ by construction, so $R_{e,i} - U
\subseteq A$.  Thus, if $U$ is finite, then $P_{e,i}$ is met.  If $U$ is
infinite, then every interval every chosen is cancelled, at which time
all of its elements enter $A$, so $U \subseteq A$.   In this case,
$R_{e,i} \subseteq (R_{e,i} - U) \cup U  \subseteq A$, so again $P_{e,i}$
is met.

Note that this construction affects $A$ only on $R_{e,i}$, so the
constructions for the various requirements operate independently, and
there is no injury.  Thus the full construction is simply a
combination of the above, over all pairs $(e,i)$.  To ensure that only
finitely many actions are taken at each stage over all $(e,i)$, one
could require that the $(e,i)$-construction act at $s$ only for
$\langle e, i \rangle \leq s$, and this would clearly not affect the
success of the individual constructions.

\end{proof}

\begin{cor}  Let $\bf a$ be a c.e.\ degree.  Then the following are equivalent:
     \begin{itemize}
          \item[(i)] $\bf a$ is not low
          \item[(ii)] There is a c.e.\ set $A$ of degree $\bf a$ such that
            $A$ has density $1$ but no computable subset of $A$ has
            density $1$.
           \item[(iii)] There is a c.e.\ set $A$ of degree $\bf a$ such that
            $A$ has density $1$ but no computable subset of $A$ has
            nonzero density.
    \end{itemize}
\end{cor}

\begin{proof} The implication (i) $\Rightarrow$ (ii) is Theorem
  \ref{nonlow}, and the implication (ii) $\Rightarrow$ (i) is Corollary
  \ref{low}.  The implication (i) $\Rightarrow$ (iii) is proved by
  combining the methods of Theorem \ref{nonlow} and Theorem
  \ref{nononzero}.  We omit the details.  The implication (iii)
  $\Rightarrow$ (ii) is immediate.
\end{proof}

In \cite{JS}, a set $A$ was defined to be \emph{coarsely computable}
if there is a computable set $B$ such that $A \triangle B$ (the
symmetric difference of $A$ and $B$) has density $0$.  It was shown in
Proposition 2.15 of \cite{JS} that there is a c.e.\ set which is
coarsely computable but not generically computable and in Theorem 2.26
of \cite{JS} that there is a c.e.\ set which is generically computable
but not coarsely computable.  The proof of the latter result is
similar to the proof that there is a c.e.\ set of density $1$ with no
computable subset of density 1 (Theorem 2.22 of \cite{JS}).  Further
the existence of a c.e.\ set which is generically computable but not
coarsely computable immediately implies the existence of a c.e.\ set of
density $1$ with no computable subset of density $1$.  Since sets of
the latter sort exist only in nonlow degrees, one might conjecture
that c.e.\ sets which are generically computable but not coarsely
computable exist only in nonlow c.e.\ degrees.  The next result refutes
this conjecture.

\begin{thm}   Every nonzero c.e.\ degree contains a c.e.\ set which is
generically computable but not coarsely computable.
\end{thm}

\begin{proof}
The proof is similar to that of Theorem 2.26 of \cite{JS}, but with
permitting added in.   Let $B$ be a noncomputable c.e.\ set.  We must
construct a c.e.\ set $A_1 \equiv_T B$ such that $A_1$ is generically
computable but not coarsely computable.   To make $A_1 \leq_T B$, we
require that if if $x$ is enumerated in $A_1$ at stage $s$, then some
$y \leq x$ is enumerated in $B$ at stage $s$.  We can then make $A_1 \equiv_T B$
by coding $B$ into $A_1$ on a computable set of density $0$, as in the proof of
Corollary \ref{exhigh}, and clearly this operation affects neither the generic
nor the coarse computability of $A_1$.  To ensure that
$A_1$ is generically computable it suffices to construct a c.e.\ set
$A_0$ such that $A_0 \cap A_1 = \emptyset$ and $A_0 \cup A_1$ has density
$1$, since the partial computable function which takes the value $0$ on
$A_0$ and $1$ on $A_1$ would then witness that $A_1$ is generically computable.
As in Corollary \ref{exhigh}, to ensure that $A_0 \cup A_1$ has density $1$, it
suffices to meet the following positive requirements:
$$P_e :   R_e \subseteq^* A_0 \cup A_1$$
Let $S_e = \{n : \varphi_e(n) = 1\}$, so that the sets $S_e$ for
$\varphi_e$ total are precisely the computable sets.  To ensure that
$A_1$ is not coarsely computable, it suffices to meet the following
negative requirements:
$$N_e : \mbox{ If $\varphi_e$ is total, then $S_e \triangle A_1$
  is not of density $0$ }$$ As in the proof of Theorem \ref{high}, the
requirements $P_e$ and $N_e$ act only on the set $R_e$.  We describe
the strategy for meeting those two requirements.  As in the proof of
that theorem, call a set $I \subseteq R_e$ \emph{large} if at least
half of the elements of $R_e$ less than $\max I$ are in $I$ .
Choose a large interval $I_0$ not containing any element already in
$A_0 \cup A_1$.  Restrain elements of $I_0$ from entering $A_0 \cup
A_1$.  Wait until the set $I_0$ becomes \emph{realized}, meaning that
$\varphi_e$ becomes defined on all of its elements.  If this never
occurs, we meet $N_e$ because $\varphi_e$ is not total.  We now appoint a
new large set $I_1$ with $\min I_1 > \min I_0$ and continue in this
fashion.  Whenever a realized set $I_j$ which has not yet intersected
$A_0 \cup A_1$ is \emph{permitted} in the sense that some number $\leq
\min I_j$ enters $B$, we force $I_j \subseteq S_e \triangle
A_1$ by enumerating all elements of $I_j \cap S_e$ into
$A_0$ and all other elements of $I_j$ into $A_1$.  Also, for all $k <
j$, if $I_k$ has not yet intersected $A_0 \cup A_1$, we enumerate all
elements of $I_k$ into $A_1$.   (Note that no permission is needed for
the latter enumerations.)   Further, for the sake of $P_e$, at each stage
$s \in R_e$, if $s$ is not restrained by $N_e$, we enumerate $s$ into $A_1$.
If infinitely many intervals are permitted as above, then we ensure
that $S_e \triangle A$ contains infinitely many large sets
and so does not have density $0$.   Suppose now that only finitely many
intervals are permitted and $\varphi_e$ is total.  Let $s_0$ be a stage
such that no interval appointed after $s_0$ is permitted after it is
realized.  Note that infinitely many intervals are appointed, and
all are realized.   Hence, we can show that $B$ is computable, since
if $I_k$ is any interval appointed after $s_0$, $B$ never changes
below $\min I_k$ after $I_k$ is realized.

Finally, $P_e$ is met because only finitely many elements of $R_e$ are
permanently restrained.  This is clear if only finitely many intervals
are ever appointed.  If infinitely many intervals are appointed then
all intervals must be realized.  Furthermore, infinitely many
intervals must be permitted after they are realized, by the above
argument.  Whenever an interval $I_k$ is permitted, we ensure that all
elements of $R_e$ less than or equal to $\max I_k$ belong to
$A_0 \cup A_1$.   Thus, if infinitely many intervals are appointed,
we have $R_e \subseteq A_0 \cup A_1$ and again $P_e$ is met.

\end{proof}

The technique we have introduced for meeting infinitary requirements
via permitting with a nonlow c.e.\ oracle has applications beyond 
the study of asymptotic density.   We illustrate this point by 
proving two theorems.

Let $A$ be a c.e.\ set.  Recall that its complement $\overline{A}$ is
called \emph{semilow} if $\{e:W_e\cap \overline{A}\ne \emptyset\}\le_T
\emptyset'$, and is called \emph{semilow$_{1.5}$} if $\{e:|W_e \cap
\overline{A}|=\infty\} \le_m \{e : |W_e| = \infty\}$. The notions of semilow and
semilow$_{1.5}$ first arose in the study of the automorphisms of the
lattice $\mathcal{E}^*$ of computably enumerable sets modulo finite sets.  Let
$\mathcal{L}^*(A)$ be the lattice of c.e.\ supersets of $A$, modulo finite sets.
Maass \cite{M} showed that, if $A$ is coinfinite, $\mathcal{L}^*(A)$ is
effectively isomorphic to $\mathcal{E}^*$ if
and only if $\overline{A}$ is semilow$_{1.5}$. Clearly the
implications low implies semilow implies semilow$_{1.5}$ hold, and it
can be shown that they cannot be reversed in general.

We prove the following.   An elegant proof not using permitting is given
in Soare's forthcoming book \cite{S1}.

\begin{theorem} Let $\bf a$ be a c.e.\ degree.   Then the following
are equivalent:
\begin{itemize}
     \item[(i)]   There is a c.e.\ set $A$ of degree $\bf a$ such that $\overline{A}$
is not semilow$_{1.5}$.
      \item[(ii)]  {\bf a} is not low.
\end{itemize}
\end{theorem}

\begin{proof} 
For the nontrivial direction, it is enough to show that 
a nonlow c.e.\ degree ${\bf a}$ bounds a non-semilow$_{1.5}$ c.e.\ set $A$ 
since then we can consider $A\oplus C$ for any c.e.\ set $C \in {\bf a}.$
Fix a c.e.\ set $C$ of degree $\bf a$.

The construction is analogous to the proof of Theorem \ref{nonlow}.
We make $A \leq_T C$ by ordinary permitting.

As before, we use the sets $R_n$, although any infinite uniformly
computable family of pairwise disjoint sets would do just as well.
We must satisfy the following conditions: 
$$Q_e:\varphi_e \mbox{ does not witness that $A$ is semilow$_{1.5}$}.$$

As before, we define a computable function $g(e,i,s)$ which threatens
to witness that $C$ is low.   Let $L_{e,i}$ be the assertion that
$C'(i) = \lim_s g(e,i,s)$.   Finally, define the \emph{requirement} 
$Q_{e,i}$ as follows:
$$Q_{e,i} : Q_e \mbox{ or } L_{e,i}$$

The requirement $Q_{e,i}$ will affect the construction only on the
set $R_{\langle e, i \rangle}$, which we denote $R_{e,i}$ for short.

For the sake of $Q_{e,i}$,  we will build sets 
$V_{e,i}=W_{h(e,i)}$,  where $h$ is computable and the index $h(e,i)$ is available during
the construction by the Recursion Theorem.  
Let $Y_{e,i} = W_{\varphi_e(h(e,i))}$ if $\varphi_e(h(e,i)) \downarrow$, and
otherwise, let $Y_{e,i} = \emptyset$.   Of course, $Y_{e,i}$ is c.e.,
uniformly in $e$ and $i$.
The construction will ensure that,
if $L_{e,i}$ fails, the following both hold:
\begin{itemize}
\item[(i)] If $Y_{e,i}$ is finite, then $V_{e,i} \cap
  \overline{A}$ is infinite.
     \item[(ii)] If $Y_{e,i}$ is infinite, then $V_{e,i} \subseteq A$.
\end{itemize}

       This suffices, since each of the above conditions implies that
$N_e$ is met.

We now describe the strategy for $Q_{e,i}$.   Initially, we define
 $g(e,i,0) = 0$ and in the construction change this from $0$ to $1$ or $1$ to $0$
only when explicitly told to, else $g(e,i,s+1) = g(e,i,s)$.  Next, we await the first
stage $s_0 > 0$ with $i \in C'[s_0]$, say with use $u_0 < s_0$.   If there is no such stage $s_0$, we satisfy
$L_{e,i}$ and hence $Q_{e,i}$ via $\lim_s g(e,i,s) = 0 = C'(i)$.  For stages $s > s_0$
we enumerate $s$ into $V_{e,i}$ if $s \in R_{e,i}$ until, if ever, we reach a stage $s_1$ such that either

(i)   some $y < u_0$ enters $C$ at $s_1$, or

(ii)  $|Y_{e,s_1}| > s_0$.

If no such stage $s_1$ exists, we satisfy $N_e$ because $Y_{e,i}$ is
finite and yet $V_{e,i} \cap \overline{A}$ is infinite.  Suppose now
that $s_1$ exists.  

If (i) occurs, we enumerate an element of $R_{e,i}$ greater than
$\max (A[s_1] \cup \{s_1\})$ into $V_{e,i}$ and restart the strategy.  

If (ii) occurs, we set $g(e,i,s_1) = 1$.  We then await a stage $s_2
> s_1$ such that some $y < u_0$ enters $C$ at $s_2$.  If no such stage
occurs, we meet $L_{e,i}$ via $\lim_s g(e,i,s) = 1 = C'(i)$.  If such
a stage occurs, we set $g(e,i,s_2) = 0$, enumerate all of
$V_{e,i}[s_2]$ into $A$,  and
restart the strategy.

We now show that this strategy succeeds in meeting $Q_{e,i}$.  If we
wait forever for some stage $s_i$ as above (in some cycle) to occur,
then $Q_{e,i}$ is met by remarks in the description of the strategy.
Suppose that we never wait in vain and so go through infinitely many
cycles.  If $Y_{e,i}$ is infinite, then (ii) occurs in infinitely many
cycles, and we ensure that $V_e \subset A$ by the action at $s_2$.  If
$Y_{e,i}$ is finite, then (ii) occurs in only finitely many cycles.
It follows that $\lim_s g(e,i,s) = 0$ because $g(e,i,s)$ changes from
$0$ to $1$ only finitely often, and after each such change it is reset
to $0$.  Also, $i \notin C'$ because in each cycle there is a stage at
which a number below the use of the computation showing $i \in C'$
enters $C$.  Hence and we meet $L_{e,i}$ via $\lim_s g(e,i,s) = 0 =
C'(i)$.  It follows that $Q_{e,i}$ is met in all cases.   As before,
the requirements don't interact, and we omit further details.

\end{proof}

The proof above can be modified for another similar property.  Cholak
\cite{CMemoir} proved a result related to Maass's by showing that if
$A$ is semilow$_2$ (a generalization of being semilow$_{1.5}$) and has
the \emph{outer splitting property} then ${\mathcal L}^*(A)$ is
isomorphic to ${\mathcal E}^*$.  $A$ has the outer splitting property
if there are two total computable functions $f$ and $g$ such that for
all $e$,
\begin{itemize}
     \item[(i)] $W_e=W_{f(e)}\sqcup W_{g(e)}$ (that is, they split $W_e$.)
     \item[(ii)] $|W_{f(e)}\cap \overline{A}|<\infty.$
     \item[(iii)] $|W_e\cap \overline{A}|=\infty$ implies $W_{f(e)}\cap 
          \overline{A}\ne \emptyset.$
\end{itemize}
Cholak and Shore showed that there is a low$_2$ c.e.\ set without the
outer splitting property \cite{CMemoir}. 
The following classifies the degrees, and
extends the result above since if $A$ is c.e.\ and $\overline{A}$ is
semilow$_{1.5}$, then $A$ has the outer splitting property \cite{CMemoir}.

\begin{thm} A c.e.\ degree ${\bf a}$ contains a c.e.\ set $A$ without
  the outer splitting property if and only if ${\bf a}$ is
  non-low. Hence having the outer splitting property is not definable
  in the lattice of c.e.\ sets.
\end{thm}

\begin{proof}  The second part of the
  statement follows from the first part and Rachel Epstein's recent result
  \cite{Ep} that there is a nonlow c.e.\ degree ${\bf c}$
  such that every c.e.\ set in that degree can be sent to a low degree
by an automorphism.

The proof of the first part is completely analogous to the proof of
the previous theorem.  We modify $Q_e$ so that it now requires us to
kill the $e$-th pair of candidates for $f$ and $g$ (say $\psi_e $ and
$\xi_e$) in the definition of the outer splitting property.  We define
a computable function $g$ as before and use it to determine the statements $L_{e,i}$
and $Q_{e,i}$ as before.  We also define sets $V_{e,i} = W_{h(e,i)}$
as witnesses for $Q_{e,i}$.  Let $Y_{e,i} = W_{\psi_e(h(e,i))}$ and
$Z_{e,i} = W_{\xi_e(h(e,i))}$. We ensure that, if $L_{e,i}$ fails to
hold, then either
\begin{itemize} 
     \item[(i)]  $Y_{e,i}, Z_{e,i}$ fail to split $W_{h(e,i)}$, or
     \item[(ii)] $Y_{e,i} \cap \overline{A}$ is infinite, or
     \item[(iii)]  $W_{h(e,i)}$ is infinite and $Y_{e,i} \subseteq A$.
\end{itemize}
This clearly suffices to meet $Q_{e,i}$.  Our strategy to achieve the
above (acting only on $R_{e,i}$) is as follows: Set $g(e,i,0) = 0$ and
wait for $s_0$ with $i \in C'[s_0]$.  At stage $s_0$, we start putting
elements of $R_{e,i}$ into $W_{h(e,i)}$.  We continue until we reach a stage
$s_1$ such that either the computation $i \in C'$ is destroyed or
$|Z_{e,i}| > s_0$.  In the former case, we dump $Y_{e,i}$ into $A$ and
restart the strategy.  In the latter case, we set $g(e,i,s_1) = 1$ and
wait for a stage $s_2$ at which the computation $i \in C'$ is injured.
At stage $s_2$, we dump $Y_{e,i}$ into $A$ and restart the strategy,
in particular setting $g(e,i,s_2) = 0$.

As before, it is easy to see that $Q_{e,i}$ is met if there are only
finitely many cycles.   In particular, if $s_1$ fails to exist, then
$W_{h(e,i)}$ is infinite and $Z_{e,i}$ is finite, so either (i) or (ii)
above holds.   If we set $g(e,i,s_1) = 1$ in infinitely many cycles,
then (iii) holds.   In the remaining case, we have $\lim_s g(e,i,s) =
0 = C'(i)$, so $L_{e,i}$ holds.

\end{proof}

\section{Arithmetical complexity of densities}

In \cite{JS}, Theorem 2.21  shows that the densities of the computable
sets are exactly the $\Delta^0_2$ reals in the interval $[0,1]$.   In this
section we obtain analogous results for c.e.\ sets in place of computable
sets, and we also study upper and lower densities as well as densities.
Throughout, we assume we have fixed a computable bijection between the
natural numbers and the rational numbers.   We say that a set of rational
numbers is $\Sigma^0_n$ if the corresponding set of natural numbers is
$\Sigma^0_n$, and similarly for other notions.  This allows us to define
what it means for a real number to be \emph{left}-$\Sigma^0_n$ as in Definition
\ref{left}.
  
We first characterize the lower densities of the computable sets.

\begin{thm} \label{s2}
  Let $r$ be a real number in the interval $[0,1]$.  Then the
  following are equivalent:
\begin{itemize}
     \item[(i)] $r = \underline{\rho}(A)$ for some computable set $A$.
     \item[(ii)]   $r$ is the lim inf of a computable sequence of rational numbers.
     \item[(iii)]   $r$ is left-$\Sigma^0_2$.
\end{itemize}
\end{thm}

\begin{proof}
  It is obvious that (i) implies (ii), since $\underline{\rho}(A) =
  \liminf_n \rho_n(A)$ by definition.

  To see that (ii) implies (iii), suppose that $r = \liminf_n q_n$,
  where $\{q_n\}$ is a computable sequence of rational numbers.  If
  $r$ is a rational number, then it is clear that (iii) holds.  Thus, we
  assume that $r$ is irrational.  Then, for all rational numbers $q$,
$$q < r \Longleftrightarrow (\forall^\infty n)[q < q_n]$$
where $(\forall^\infty n)$ means ``for all but finitely many $n$.''
This follows easily from the definition of the lim inf, using the
assumption that $r$ is irrational to prove the implication from right
to left.  Expanding the right-side of the above equivalence shows that
$r$ is left-$\Sigma^0_2$.

We now show that (iii) implies (ii).  We begin with a lemma which
characterizes $\Sigma^0_2$ sets of natural numbers in terms of
computable approximations.  The following lemma improves the standard
result that for every $\Sigma^0_2$ set $A$ there are uniformly
computable sets $A_s$ such that, for all $k$, $k \in A$ if and only if
$(\forall^\infty s)[k \in A_s]$.

\begin{lem} \label{s2approx}
Let $A$ be a $\Sigma^0_2$ set.  Then there is a uniformly
  computable sequence of sets $\{A_s\}$ such that
     \begin{itemize}
     \item[(i)] For all $k \in A$, we have $k \in A_s$ for all sufficiently
       large $s$ 
           \item[(ii)]  There exist infinitely many $s$ such that $A_s \subseteq A$
     \end{itemize}
\end{lem}

\begin{proof} (of lemma)
     To prove this lemma, we use the Lachlan-Soare ``hat trick'' (\cite{S},
     page 131), with which we assume the reader is familiar.  Since
     $A$ is c.e.\ in $\bf 0'$ there exists an $e$ such that $A$ is the
     domain of $\Phi_e^K$.  Now let $A_s = \{k :
     \hat{\Phi}_{e,s}(K_s,k) \downarrow\}$.  Then if $k \in A$, then
     $\Phi_e^K(k) \downarrow$, and so $\hat{\Phi}_{e,s}(K_s, k)
     \downarrow$ for all sufficiently large $s$.  Now let $T$ be the
     set of true stages.  $T$ is infinite.  If $s \in T$ and
     $\hat{\Phi}_{e,s}(K_s,k) \downarrow$, then $\Phi_e^K(k)
     \downarrow$.  It follows that $A_s \subseteq A$ for all $s$ in
     $T$.
\end{proof}

We now show that (iii) implies (ii) in the theorem.  Let $r$ be a real
number which is left-$\Sigma^0_2$.  We must produce a computable
sequence $\{q_s\}$ of rational numbers such that $\liminf_s q_s = r$.
Let $A = \{q \in \mathbb{Q} : q < r\}$, so that $A$ is $\Sigma^0_2$ by
hypothesis.  Let the uniformly computable sets $A_s$ be related to $A$
as in the lemma.  By truncating the sets if necessary, we may assume
that every rational in $A_s$ corresponds to a natural number $<s$
under our coding of rationals.  Thus $A_s$ is finite, and we may
effectively compute a canonical index for the finite set of natural
numbers corresponding to it.  Let $q_s$ be the greatest rational
number in $A_s$ if $A_s$ is nonempty, and otherwise let $q_s = 0$.
Then $\{q_s\}$ is a computable sequence of rational numbers.  By
hypothesis, there are infinitely many $s$ such that $A_s \subseteq A$
and thus every element of $A_s$ is less than $r$.  Using the
definition of $q_s$ and the hypothesis that $r \geq 0$, it follows
that there are infinitely many $s$ such that $q_s \leq r$, and thus
$\liminf_s q_s \leq r$.  To show that $r \leq \liminf_s q_s$, note that if $q <
r$, then $q \in A$, so $q \in A_s$ for all sufficiently large $s$, so
$q \leq q_s$ for all sufficiently large $s$.  It follows that every
rational number $q < r$ is $\leq \liminf_s q_s$, so $r \leq \liminf_s
q_s$.  This completes the proof that (iii) implies (ii).

We complete the proof of the Theorem by showing that (ii) implies (i).
Thus we must show that if $r \in [0,1]$ and $r = \liminf_s q_s$ where
$\{q_s\}$ is a computable sequence of rationals, then there is a
computable set $A$ such that $\underline{\rho}(A) = r$.  Essentially,
this follows from the proof of Theorem 2.21 of \cite{JS}, where it was
shown that every $\Delta^0_2$ real in $[0,1]$ is the density of a
computable set.  For the convenience of the reader and for use in
Corollary \ref{lsp}, we put in the
details in the following lemma.

\begin{lem} \label{infsup}
  Let $\{q_s\}$ be a computable sequence of rational numbers such that
  $0 \leq \liminf_s q_s$ and $\limsup_s q_s \leq 1$.  Then there is a
  computable set $A$ such that $\underline{\rho}(A) = \liminf_s q_s$
  and $\overline{\rho}(A) = \limsup_s q_s$.
\end{lem}

\begin{proof}
  First, we may assume that each $q_s$ lies in the
interval $(0,1)$, by replacing $q_s$ by $1/(s+1)$ if $q_s \leq 0$ and
by $1 - 1/(s+1)$ if $q_s \geq 1$ and otherwise leaving $q_s$
unaltered.  Since $r \in [0,1]$, the resulting sequence of rationals
still has $r$ as its lim inf.  We define a computable set $A$ and an
increasing sequence $\{s_n\}$ of natural numbers such that, for all $n
$:
\begin{itemize}
      \item[(i)] $|\rho_{s_n}(A) - q_n| \leq 1/(n+1)$ 
      \item[(ii)] For all natural numbers $k$ in the interval $(s_n,
        s_{n+1})$, $\rho_k(A)$ is between $\rho_{s_n}(A)$ and
        $\rho_{s_{n+1}}(A)$.
\end{itemize}

Let $s_0 = 1$ and put $0$ into $A$.  Now assume inductively that $s_n$
and $A \upharpoonright s_n$ are defined, so that $\rho_{s(n)}(A)$ is
defined.  There are now two cases.    

If $\rho_{s_n}(A) > q_{n+1}$, let $s_{n+1}$ be the least number $t >
s_n$ such that $\rho_t (A \upharpoonright s_n) \leq q_{n+1}$.  (Such a
$t$ exists because $q_{n+1} > 0$ and $\rho_t(A \upharpoonright s_n)$
approaches $0$ as $t$ approaches infinity.)   Let $A \upharpoonright
s_{n+1} = A \upharpoonright  s_n$.

If $\rho_{s_n}(A) \leq q_{n+1}$, let $s_{n+1}$ be the least number $t > s_n$
such that $\rho_t ((A \upharpoonright s_n) \cup [s_n, t)) \geq q_{n+1}$, and
let $A \upharpoonright s_{n+1} = A \upharpoonright s_n \cup [s_n, s_{n+1})$.

To verify (i), use that $s_n \geq n$ and the minimality of $t$ in each case.
To verify (ii), use that the interval $[s_n, s_{n+1})$ is either contained in
or disjoint from $A$, so that $\rho_t(A)$ is either increasing or decreasing
in $t$ on this interval.  We omit the details.

It remains to show that $\underline{\rho}(A) = \liminf_s \rho_s(A) =
\liminf_n q_n$.  Since $\lim_{n}(q_n - \rho_{s(n)}) = 0$ and
$\{\rho_{s(n)}(A)\}$ is a subsequence of $\{\rho_s(A)\}$, we have
$\liminf_s \rho_s(A) \leq \liminf_n \rho_{s(n)} = \liminf_n q_n$.  To
show that $\liminf_n \rho_{s(n)}(A) \leq \liminf_s \rho_s(A)$, note
that for every $k > 0$ there is a number $t(k)$ such that
$\rho_{s(t(k))}(A) \leq \rho_k(A)$, namely if $s(n) \leq k < s(n+1)$, let
$t(k)$ be $n$ if $\rho_{s(n)} \leq \rho_k(A)$, and otherwise let
$t(k)$ be $n + 1$.  Further, by this definition, the function $t$ is
finite-one, and hence $t(k)$ approaches infinity as $k$ approaches
infinity.  We thus have $\liminf_k \rho_k(A) \leq \liminf_n
\rho_{s(n)}(A)$, which completes the proof of the Lemma.
\end{proof}

The theorem follows.
\end{proof}

\begin{cor} \label{pi2}
  Let $r$ be a real number in the interval $[0,1]$.  Then the
  following are equivalent:
\begin{itemize}
     \item[(i)]  $r$ is the upper density of some computable set.
     \item[(ii)]  $r$ is left-$\Pi^0_2$
\end{itemize}
\end{cor}

\begin{proof} Note that $\overline{\rho}(A) = 1 -
  \underline{\rho}(\overline{A})$ for every set $A$, and that for
  every real number $r$, $1-r$ is left-$\Sigma^0_2$ if and only if $r$
  is left-$\Pi^0_2$.  Since the computable sets are closed under
  complementation, the corollary follows.
\end{proof}

The following corollary sums up our results on upper and lower densities
of computable sets.

\begin{cor} \label{lsp}
      Let $a$ and $b$ be real numbers such that $0 \leq a \leq b \leq 1$.
Then the following are equivalent:
\begin{itemize}
      \item[(i)]  There is a computable set $R$ with lower density $a$ and upper
density $b$
        \item[(ii)]   $a$ is left-$\Sigma^0_2$ and $b$ is left-$\Pi^0_2$.
\end{itemize}
\end{cor}

\begin{proof}
    It follows at once from Theorem \ref{s2} and Corollary \ref{pi2} that
(i) implies (ii).  For the converse, assume that (ii) holds of $a$ and $b$.
Since $a$ is left-$\Sigma^0_2$, by Theorem \ref{s2}, there is a computable
sequence of rationals $\{q_n\}$ with $\liminf_n q_n = a$.   Since $b$ is
left-$\Pi^0_2$, by the proof of Corollary \ref{pi2}, there is a computable
sequence of rationals $\{r_n\}$ with $\limsup_n r_n = b$.

If $a = b$, then $a$ is a $\Delta^0_2$ real and hence is the density
of a computable set by Theorem 2.21 of \cite{JS}.  Thus, we may assume
that $a < b$.  Fix a rational number $q^*$ such that $a \leq q* \leq
b$.  By replacing $q_n$ by $\max\{q_n, q^*\}$, we may assume that $q_n
\leq q^*$ for all $n$, and hence $\limsup_n q_n \leq q^* \leq b$.
Similarly, we may assume that $\liminf_n r_n \geq a$.  Now define a
computable sequence of rationals $\{s_n\}$ by $s_{2n} = q_n$ and
$s_{2n+1} = r_n$.  Then $\liminf_n s_n = \min \{\liminf_n q_n ,
\liminf_n r_n\} = a$ and $\limsup_n s_n = \max \{\limsup_n q_n ,
\limsup_n r_n\} = b$.  The corollary now follows by applying Lemma
\ref{infsup} to the sequence $\{s_n\}$.
\end{proof}

We now consider the complexity of the various kinds of density
associated with c.e.\ sets.   The first result follows easily from what we have
done for computable sets.

\begin{thm} \label{p2ce} Let $r$ be a real number in the interval $[0,1]$.  Then the following are
equivalent:
\begin{itemize}
      \item[(i)]   $r$ is the upper density of a c.e.\ set.
      \item[(ii)]   $r$ is left-$\Pi^0_2$.
\end{itemize}
\end{thm}

\begin{proof}
  It follows immediately from Corollary \ref{pi2} that (ii) implies
  (i).  To show that (i) implies (ii), let $r$ be the upper density of
  a c.e.\ set $A$.  We may assume without loss of generality that $r$
  is irrational.  Then for $q$ rational, we have
$$q < r  \Longleftrightarrow (\exists^\infty n) [q < \rho_n(A)]$$
Since the predicate $q < \rho_n(A)$ is $\Sigma^0_1$, (ii) follows.
\end{proof}

\begin{thm}  Let $r$ be a real number in the interval $[0,1]$.  Then the following are
equivalent:
\begin{itemize}
      \item[(i)]   $r$ is the lower density of a c.e.\ set.
      \item[(ii)]   $r$ is left-$\Sigma^0_3$.
\end{itemize}
\end{thm}

\begin{proof}
  By relativizing Theorem \ref{s2} to $0'$ and applying Post's
  Theorem, we see that, for $r \in [0,1]$, $r$ is the lower density of
  a $\Delta^0_2$ set if and only if $r$ is left-$\Sigma^0_3$.   It follows
immediately that (i) implies (ii) above.    To prove the converse, it suffices
to show that for every $\Delta^0_2$ set $B$ there is a c.e.\ set $A$ such that
$A$ and $B$ have the same lower density.   Let the $\Delta^0_2$ set $B$ be
given.   We will give a strictly increasing $\Delta^0_2$ function $t$ and a
c.e.\ set $A$ such that, for all $n$, $\rho_{t(n)}(A) = \rho_n(B)$.    This
implies that $\underline{\rho}(B) \geq \underline{\rho}(A)$.   To obtain the
opposite inequality (and hence $\underline{\rho}(A) = \underline{\rho}(B)$),
we require further that, for each $n$,  $A \cap [t(n), t(n+1))$ be an initial
segment of the interval $[t(n), t(n+1))$, so that the least value of
$\rho_k(A)$ for $k \in [t(n), t(n+1))$ occurs when $k = t(n)$ or $k = t(n+1)$. 
It then follows that    $ \liminf_k  \rho_k(A) \geq \liminf_n \rho_{t(n)}(A)$.
Hence, $\underline{\rho}(B) = \liminf_n \rho_n(B) 
= \liminf_n \rho_{t(n)}(A)   \leq \liminf_k  \rho_k(A) = \underline{\rho}(A)$.

The following straightforward lemma will be useful in defining $t$ as described
above.

\begin{lem} \label{rat}
  Let $F \subseteq \omega$ be a finite set, $a,d \in \omega$, and $r$ be
a rational number in the interval $(0,1)$.   Then there is a finite set $G \supseteq F$
and $c \in \omega$ such that:
     \begin{itemize}
         \item[(i)] $G \upharpoonright a = F \upharpoonright a$
         \item[(ii)] $c > d$
         \item[(iii)] $\rho_c(G) = r$
         \item[(iv)] $G \cap [a, \infty)$ is an initial segment of $[a, \infty)$.
     \end{itemize}
\end{lem}

\begin{proof}
  For any $b$, let $G_b = F \cup [a,b)$.  Then for every sufficiently
  large $b$, we have $\rho_b(G_b) > r$, since $\lim_b \rho_b(G_b) = 1
  > r$.  We will set $G = G_b$ for a suitable choice of $b$.  Then (i)
  above holds with $G = G_b$ for all $b$, and (iv) above holds with $G
  = G_b$ for all $b > \max(F)$.  To make (ii) and (iii) hold, we set $c =
  |G_b|/r$, where $b$ is chosen so that $|G_b|/r$ is an integer
  greater than $d$, $b > a$ and $b > \max(F)$.  To see that such a $b$
  exists (and in fact infinitely many such $b$ exist), note that there
  is a constant $k$ such that $|G_b| = b - k$ for all sufficiently
  large $b$.  For any sufficiently large such $b$, we have $c = |G_b|/r > b > \max F$, so
  $c > \max(G_b)$.  Hence $\rho_c(G_b) = |G_b| / c = r$, and therefore
  (i)-(iv) all hold with $G = G_b$ and $c = |G_b|/r$.
\end{proof}

We now define a c.e.\ set $A$ and a strictly increasing $\Delta^0_2$
function $t$ as described above.  We enumerate $A$ effectively and
obtain $t$ as $\lim_s t(n,s)$, where $t(.,.)$ is computable.  Let
$\{B_s\}$ be a computable approximation to $B$.  At stage $0$, let
$A_0 = \emptyset$, and $t(n,0) = n+1$ for all $n$.  We also make the
convention that $t(-1,s) = 0$ for all $s$.  At stage $s+1$, suppose
inductively that $A_s$ and all values of $t(n,s)$ have been defined.
Let $n_s$ be the least $n \leq s$ such that $\rho_{t(n,s)}(B_s) \neq
\rho_n(A_s)$, provided such an $n$ exists.  (If no such $n$ exists,
let $A_{s+1} = A_s$ and $t(n,s+1) = t(n,s)$ for all $n$.)  Assuming
$n_s$ exists, apply Lemma \ref{rat} with $F = A_s$, $a = t(n_s-1,s)$
and $d = \max (A_s \cup \{t(n_s,s)\})$ to obtain a finite set $G
\supseteq A_s$ and a number $c > \max (A_s \cup \{t(n_s,s)\})$ such that
$\rho_c(G) = \rho_n(A_s)$, $G \upharpoonright t(n_s-1,s) = 
A_s \upharpoonright t(n_s-1,s)$, and $G \cap [t(n_s-1,s) , \infty)$ is an
initial segment of $[t(n_s-1,s) , \infty)$.  Let $A_{s+1} = G$ and
$t(n_s,s+1) = c$.  (To apply Lemma \ref{rat} we need $0 < \rho_n(A_s)
< 1$, so if $\rho_n(A_s) = 0$, replace it by $1/(n+1)$, and if
$\rho_n(A_s) = 1$, replace it by $1 - 1/(n+1)$.  In the limit, these
replacements have no effect.)  For $m < n_s$, define $t(m,s+1) =
t(m,s)$, and for $m > n_s$ define $t(m,s+1) = c + m - n_s$.

We now show that, for each $k > 0$ that there are only finitely many
$s$ with $n_s = k$, $\lim_s t(k, s)$ exists, and, if $t(k)$ is this
limit, $\rho_{t(k)}(A) = \rho_k(B)$.  This result is proved by
induction on $k$, so assume it holds for all $m < k$.  Let $b$ be a
stage $\geq k$ such that, for all $m < k$ and all $s \geq b$,
$t(m,s) = t(m)$, $n_s \neq m$, $\rho_{t(m,s)}(A_s) = \rho_m(B_s)$, and
$B_s \upharpoonright k = B \upharpoonright k$.  If
$\rho_{t(k,b)}(A_b) \neq \rho_k(B_{b})$, we set $n_{b} = k$, and
the construction ensures that $\rho_{t(k,b+1)}(A_b) = \rho_k(B_b)$.
Then, by construction, there is no $s > b$ with $n_s = k$.  It
follows that there are only finitely many $s$ with $n_s = k$, and that
$\lim_s t(k,s) = t(k)$ exists.  It also follows that $\rho_{t(k)}(A) =
\rho_k(B)$, since this holds at stage $b + 1$ (whether or not
$n_{b} = k$) and is preserved forever thereafter.   Finally, we show
by induction on $s$ that, for all $k$, $A_s \cap [t(k,s), t(k+1,s))$ is an initial segment
of the interval $[t(k,s), t(k+1,s)$.  This obviously holds for $s = 0$.
Now assume it holds for $s$ and that $n_s$ exists.   For $k < n_s - 1$, it holds
for $s+1$ by the inductive hypothesis and the choice of $a$ as
$t(n_s-1, s)$.   For $k = n_s  - 1$, it holds for $s+1$ by the choice of $G$ in the construction.
For $k \geq n_s$, it holds vacuously at stage $s+1$ because $A \cap [t(k,s), t(k+1,s))$
is empty.  By taking the limit as $s$ approaches infinity, it follows that, for all $k$,
$A \cap [t(k), t(k+1))$ is an initial segment of $[t(k), t(k+1))$.
This completes the proof of the theorem.
\end{proof}

\begin{defn}   Let $A \subseteq \omega$ be a set and let $r$ be a
real number in the unit interval.    Then $A$ is \emph{computable at density $r$} if there is
a partial computable function $\varphi$ such that $\varphi(n) = A(n)$ for
all $n$ in the domain of $\varphi$ and the domain of $\varphi$ has lower
density greater than or equal to  $r$. 
\end{defn}

Note that  a set $A$ is generically computable if and only if $A$ is computable at density $1$.
So a first  natural question is to find sets which are computable at all densities less than $1$
but which are not generically computable.   We thank Asher Kach for greatly simplifying our 
original proof of this result.

\begin{obs} Every nonzero Turing degree contains a set $A$ which is computable at all densities $r < 1$
but which is not generically computable.
\end{obs}

\begin{proof}  We have seen that the set $\mathcal{R}(A)$ is generically computable if and only if $A$ is computable.
Every set of the form $\mathcal{R}(A)$ is computable at all densities $r < 1$:   For a given $t \in \mathbb{N}$, 
take the  finite list of which $m \le t$ are in $A$.  Given this list we can answer 
correctly on $\bigcup_{m \le t} R_m$, which has density $1 - 2^{-(t+1)}$ and 
not answer on any $\mathcal{R}_k$ with $k > t$.  This argument is, of course, completely nonuniform.
\end{proof}

A second nautral  question is: For what $r \in [0,1]$ is there a set which is
computable at density $r$ but not at any higher density?  The above theorem
answers the question.

\begin{cor} \label{nohigherdensity} Let $r \in [0,1]$.  There is a set
  $A$ which is computable at density $r$ but not at any higher density
  if and only if $r$ is left-$\Sigma^0_3$.
\end{cor}
 
\begin{proof}    
  If  $r \in [0,1]$ is left-$\Sigma^0_3$ there is a c.e.\ set $C$ with lower density $r$.
Let $S$ be a simple set of density $0$, which exists by the proof of Proposition 2.15
of \cite{JS}.   Define $A = C \cup S$.
Then $A$ is computable at density $r$ but  $A$ cannot be computable
at a density $r' > r$ since if there were a partial algorithm $\varphi$ for $A$ 
whose domain $D$ had lower density $r' > r$ then  $D \cap \{n : \varphi(n) = 0\}$ would  
be an infinite  c.e.\ set  disjoint from  $S$, contradicting that $S$ is simple.

  If $r$ is not  left-$\Sigma^0_3$ and $A$ is computable at density $r$, then
there is a partial algorithm for $A$ whose domain has lower density $r' \ge r$.
Since $r'$ is  left-$\Sigma^0_3$ but $r$ is not,  we have $r' > r$.
\end{proof}
  
It remains to consider the densities of c.e.\ sets.  Note that if $A$
is c.e., then $\rho_n(A) = \lim_s g(n,s)$ where $g$ is a computable
function taking rational values, $g(n,s) \leq g(n,s+1)$ for all $n$
and $s$, and for each $n$ there are only finitely many $s$ such that
$g(n,s) \neq g(n,s+1)$, namely $g(n,s) = \rho_n(A_s)$, where $\{A_s\}$
is a computable enumeration of $A$.  Hence, $\underline{\rho}(A) =
\liminf_n \lim_s g(n,s)$ and $\overline{\rho}(A) = \limsup_n \lim_s
g(n,s)$.  The next result turns this around to show how computable
functions $g$ with these stability and monotonicity properties can be
used to produce c.e.\ sets with corresponding upper and lower
densities.

\begin{thm} \label{double}
Let $g : \omega^2 \to \mathbb{Q} \cap (0,1)$  be a computable function
such that:
\begin{itemize}
      \item[(i)]  $g(n,s) \leq g(n,s+1)$ for all $n$ and $s$, and
      \item[(ii)] $\{s : g(n,s) \neq g(n,s+1)\}$ is finite for all $n$.
\end{itemize}
      Let $h(n) = \lim_s g(n,s)$, so $h : \omega \to \mathbb{Q}$ is
total by (ii).   Then there is a c.e.\ set $A$ such that:
\begin{itemize}
        \item[(iii)]  $\underline{\rho}(A) = \liminf_n h(n)$, and
        \item[(iv)]  $\overline{\rho}(A) = \limsup_n h(n)$.
\end{itemize}
\end{thm}

\begin{proof}
  By changing $g(n,s)$ by at most $1/n$ we may assume without loss of
  generality that $g(n,s)$ is an integer multiple of $1/n$ for all $s$
  and all $n > 0$.  Of course, this changes $h(n)$ by at most $1/n$
  and has no effect on $\liminf_n h(n)$ or $\limsup_n h(n)$.
  Partition each interval $[n!, (n+1)!) $ into disjoint subintervals
  $I_{n,1}, I_{n,2}, \dots, I_{n, r_n}$ of size $n$, where $r_n =
  ((n+1)! - n!)/n = n!$.  From each interval $I_{n,i}$ enumerate exactly $n
  h(n) = n \max_s g(n,s)$ numbers into the c.e.\ set $A$.  Note that
  this can be done effectively since $g$ is computable, and $nh(n)$ is
  an integer not exceeding $n$.  Define the density of a set $A$ on a
  nonempty finite interval $I$ to be $|A \cap I|/|I|$.  Thus we have
  ensured that the density of $A$ on each interval $I_{n,i}$ for $1
  \leq i \leq r_n$ is $h(n)$.  From this, it is easily seen that, modulo
  error terms which approach $0$ as $n$ approaches infinity,
  $\rho_{(n+1)!}(A) = h(n)$ and if $k \in [n!, (n+1)!)$, then
  $\rho_k(A)$ is between $h(n-1)$ and $h(n)$.  We then get that
  $\underline{\rho}(A) = \liminf_k \rho_k(A) = \liminf_n
  \rho_{(n+1)!}(A) = \liminf_n h(n)$, and similarly for
  $\overline{\rho}(A)$.

  We now spell out the details of the above approximations.  It is
  easy to see that if $I_1, I_2, \dots, I_t$ are disjoint intervals,
  then the density of $A$ on $I_1 \cup I_2 \cup \dots \cup I_t$ is at least
  the minimum of the density of $A$ on the intervals $I_1, \dots ,
  I_t$, and at most the maximum of the density of $A$ on these
  intervals.  Hence, the density of $A$ on the interval $[n!, (n+1)!)$
  is $h(n)$, since the density of $A$ on each subinterval $I_{n,i}$ is
  $h(n)$.  Using this to calculate the cardinality of $A \cap [n!,
  (n+1)!)$ and noting that $0 \leq |A \cap [0,n!)| \leq n!$, it
  follows that
$$h(n)((n+1)!-n!) \leq |A \cap [0,(n+1)!)| \leq n! + h(n)((n+1)!-n!)$$
Dividing through by $(n+1)!$ yields that
$$h(n) - \frac{h(n)}{n+1} \leq \rho_{(n+1)!}(A) \leq \frac{1}{n+1} + h(n) - \frac{h(n)}{n+1}$$
and hence
$$-\frac{h(n)}{n+1} \leq \rho_{(n+1)!}(A) - h(n) \leq \frac{1}{n+1} -
\frac{h(n)}{n+1}$$
It follows that $\lim_n (\rho_{(n+1)!}(A) - h(n)) = 0$.

We now show that if $k \in (n!, (n+1)!]$, then $\rho_k(A)$ is between
$h(n-1)$ and $h(n)$ with an error term which approaches $0$ as $n$
approaches infinity.  Consider first the case where $k$ has the form
$\max(I_{n,i}) + 1$ for some $i$.  Then $[0,k)$ is the disjoint union
of $[0, n!)$ and the intervals $I_{n,j}$ for $j \leq i$.  Hence $\rho_k(A)$
is between $\rho_{n!}(A)$ and $h(n)$, and, as we have noted, $\lim_n
(\rho_{n!}(A) - h(n-1)) = 0$.  Since the intervals of $I_{n,k}$ have
size $n$, every $c \in (n!, (n+1)!]$ differs from at most $n$ by a
number of the form $\max(I_{n,i}) +1$.  Finally if $a,b \geq n!$ and
$|a-b| \leq n$, we have that $|\rho_a(A) - \rho_b(A)| \leq
(n+1)/(n-1)!$.  To see this, let $u = |A \upharpoonright a|$ and $v =
|A \upharpoonright b|$.  Since $|a - b| \leq n$, we also have $|u - v|
\leq n$.  Note that $\rho_a(A) - \rho_b(A) = u/a - v/b$.  Thus, it
suffices to show that both  $v/b - u/a$ and $u/a - v/b$ are less than
or equal to $(n+1)/(n-1)!$.  We may assume without loss of generality
that $a \leq b$ and hence $u \leq v$.  We have
$$\frac{v}{b} - \frac{u}{a} \leq \frac{v}{a} - \frac{u}{a} = \frac{v-u}{a} 
\leq \frac{n}{n!} \leq \frac{n+1}{(n-1)!}$$
since $0 < a \leq b$ and $v - u \leq n$.
Also,
$$\frac{u}{a} - \frac{v}{b} \leq \frac{u}{a} - \frac{u}{b} = \frac{u}{ab}(b-a)
\leq \frac{(n+1)!}{(n!)^2}n = \frac{n+1}{(n-1)!}$$
since $b > 0$, $u \leq v$, $a, b \geq n!$, and $b - a \leq n$.

Hence $\underline{\rho}(A) = \liminf_a \rho_a(A) = \liminf_n h(n)$
and similarly for $\overline{\rho}(A)$.
\end{proof}

\begin{thm} 
Let $r$ be a real number in the interval $[0,1]$.   Then the following
are equivalent:
\begin{itemize}
     \item[(i)]  $r$ is the density of some c.e.\ set.
     \item[(ii)]   $r$ is left-$\Pi^0_2$.
\end{itemize}
\end{thm}

\begin{proof}
      The implication (i) implies (ii) is immediate from Theorem \ref{p2ce},
which implies that the upper density of every c.e.\ set is a left-$\Pi^0_2$
real.

For the implication (ii) implies (i), assume that $r$ is 
left-$\Pi^0_2$.  Then by the dual of Theorem \ref{s2}, there is a
computable sequence of rationals $\{q_n\}$ such that $r = \limsup_n
q_n$ and $0 \leq q_n \leq 1$ for all $n$.  We now define a computable
function $g: \omega^2 \to \mathbb{Q} \cap [0,1]$ to which to apply
Theorem \ref{double}.  We define $g(n,s)$ by recursion on $s$.  Let
$g(n,0) = 0$.  For the inductive step, define

\begin{equation*}
g(n,s+1) = 
\begin{cases} q_s  &  \text{if $q_s \geq g(n,s) + \frac{1}{n+1}$ and 
$s \geq n$}
  \\
  g(n,s) & \text{otherwise}
\end{cases}
\end{equation*}

Clearly, $g$ satisfies the hypotheses of Theorem \ref{double}, so by that
result there is a c.e.\ set $A$ such that $\underline{\rho}(A) =
\liminf_n h(n)$ and $\overline{\rho} (A) = \limsup_n h(n)$, where
$h(n) = \lim_s g(n,s)$.  Thus, it suffices to show that $\lim_n h(n) =
\limsup_n q_n$. To this end, define $b(n) = \sup_{s \geq n} q_s$ and note
that $\limsup_n q_n = \lim_n b(n)$.   By the definition of $g$ and the fact
that $h(n) = \lim_s g(n,s)$, we have
$$b(n) - \frac{1}{n+1} \leq h(n) \leq b(n)$$
for all $n$.   It follows that $\lim_n h(n) = \lim_n b(n) = \limsup_n q_n$. 

\end{proof}

\section{Density and immunity properties}

In computability theory, a whole spectrum of immunity type properties
has been studied, with the weakest being immunity itself and the
strongest one commonly studied being cohesiveness.  In this section,
we study results relating immunity properties and asymptotic density.
It was already shown in the proof of Proposition 2.l5 of \cite{JS}
that there is a simple set of density $0$, and hence an immune set of
density 1.  We observe in this section, for example, that every
hyperimmune set has lower density $0$, every strongly hyperhyperimmune
(shhi) set has upper density less than $1$, and that every cohesive
set has density $0$.  We also prove contrasting results -- for example
shhi sets can have upper density arbitrarily close to $1$.

We begin by reviewing the definitions of these properties, which are
standard.

\begin{defn}   Let $A \subseteq \omega$ be an infinite set.
     \begin{itemize}
          \item[(i)]   The set $A$ is \emph{immune} if $A$ has no infinite
c.e.\ subset.
           \item[(ii)]   The set $A$ is \emph{hyperimmune}  if there
is no computable function $f$ such that the sets $D_{f(0)}, D_{f(1)},
\dots$ are pairwise disjoint and all intersect $A$. (Here $D_n$  is the finite set with canonical
index $n$.)
\item[(iii)] The set $A$ is \emph{hyperhyperimmune}, or \emph{hhi}, if there
  is no computable function $f$ such that the sets $W_{f(0)},
    W_{f(1)}, \dots$ are pairwise disjoint, finite, and all intersect
    $A$.
          \item[(iv)]  The set $A$ is \emph{strongly hyperhyperimmune}, or \emph{shhi}, if there is no
computable function $f$ such that the sets $W_{f(0)}, W_{f(1)}, \dots$ are pairwise disjoint
and all intersect $A$.   (Thus the finiteness requirement on $W_{f(n)}$ is dropped here.)
         \item[(v)]   The set $A$ is \emph{r-cohesive} (respectively \emph{cohesive}) if there
is no computable (respectively c.e.) set $S$ such that $A \cap S$ and $A \cap \overline{S}$ are
both infinite.
      \end{itemize}

It is well known that each property above (except immunity) implies the one before it, and that these
implications are proper.   For more information, see, for example, Chapter XI.1 of \cite{S}.

\end{defn}
\begin{thm}
     \begin{itemize}
         \item[(i)]  Every hyperimmune set has lower density $0$.
         \item[(ii)]  There is a co-c.e.\ hyperimmune set with upper density $1$.
      \end{itemize}
\end{thm}
\begin{proof} (Sketch) Let $I_n$ be the interval $[n!, (n+1)!)$.  If
  $A$ is hyperimmune, then $A \cap I_n = \emptyset$ for infinitely
  many $n$, from which it follows that $\underline{\rho}(A) = 0$.  For
  the second part, it is a straightforward finite injury argument to
  construct a c.e.\ co-hyperimmune set $B$ such that $B \cap I_n = \emptyset$
  for infinitely many $n$, so that $\overline{B}$ is the desired
  co-c.e.\ hyperimmune set with upper density $1$.
\end{proof}
 
\begin{thm}
      \begin{itemize}
          \item[(i)] Every co-c.e.\ hhi set has density $0$.
          \item[(ii)]  Every $\Delta^0_2$ hhi set has upper density less than $1$.
          \item[(iii)] There is a hhi set with upper density $1$.
       \end{itemize}  
\end{thm}
\begin{proof} For the first part, recall that by \cite{M} every
  co-c.e.\ hhi set $A$ is dense immune, i.e. the principal function of
  $A$ dominates every computable function.  From this it easily
  follows that $A$ has density $0$.  For the second part, use the
  known fact (see the lemma below) that every $\Delta^0_2$ hhi set is shhi, and apply the
  first part of the next theorem.  For the third part, note that every
  $2$-generic set is hhi and has upper density $1$.

       The lemma below is due to S. B. Cooper \cite{C}, and we include a proof for the convenience of
the reader.

\begin{lem} (\cite{C}) If $A$ is both $\Delta^0_2$ and hhi, then $A$ is shhi.
\end{lem}

\begin{proof}  Let $\{A_s\}$ be a computable approximation to $A$,
  and suppose that $A$ is infinite and not shhi.  We prove that $A$ is
  not hhi.  Let $\{U_n\}$ be a weak array witnessing that $A$ is not
  shhi, so the sets $U_n$ are uniformly c.e., pairwise disjoint, and
  all intersect $A$.  To show that $A$ is not hhi, it suffices to
  produce uniformly c.e.\ sets $\{V_n\}$ with each $V_n$ a finite
  subset of $U_n$ so that each $V_n$ intersects $A$.  Let $V_{n,s}$ be
  the set of numbers enumerated in $U_n$ before stage $s$, and define
  $U_{n,s}$ analogously.  At each stage $s$, if $V_{n,s} \cap A_s =
  \emptyset$, let $V_{n,s+1} = V_{n,s} \cup U_{n,s}$, and otherwise
  let $V_{n,s+1} = V_{n,s}$.

Clearly, $V_n \subseteq U_n$.  Assume for a contradiction that $V_n$ 
is infinite.   Then $V_n = U_n$ so $V_n \cap A \neq \emptyset$.   It
follows that $V_{n,s} \cap A_s \neq \emptyset$ for all sufficiently large
$s$, so $V_n$ is finite, which is the desired contradiction.   Hence
$V_n$ is finite.   Now assume for a contradiction that $V_n \cap A = \emptyset$.
Then $V_n \cap A_s = \emptyset$ for all sufficiently large $s$, and hence
$V_n = U_n$.   It follows that $V_n \cap A \neq \emptyset$, which is the
desired contradiction.
\end{proof}

\end{proof}

\begin{thm} \label{shhi}
     \begin{itemize}
          \item[(i)] No shhi set has upper density $1$.
          \item[(ii)] For every $\epsilon > 0$ there is a shhi set with
            upper density at least $1 - \epsilon$.
      \end{itemize}
\end{thm}

\begin{proof}
  For (i), let $A$ be shhi, and consider the sets $\{R_n\}$ where, as
  usual, $R_n = \{k : 2^n \mid k \ \& \  2^{n+1} \nmid k \}$.  Since these
  sets are pairwise disjoint and uniformly computable, there exists $n$ such
  that $R_n \cap A = \emptyset$.  Since $\rho(R_n) > 0$, it follows
  that $\overline{\rho}(A) < 1$.

  For (ii) we use a special kind of Mathias forcing.  Let $q_0$ be a
  rational number such that $1 - \epsilon < q_0 < 1$.  Let $P$ be the
  set of pairs $(F,I)$ where $F, I$ are subsets of $\omega$, $F$ is
  finite, $I$ is infinite, $F \cap I = \emptyset$, and
  $\overline{\rho}(I) > q_0$.  Thus, we are using Mathias forcing with
  conditions of upper density strictly greater than $q_0$.  If $(F,I) \in
  P$, say that $A$ \emph{satisfies} $(F,I)$ if $F \subseteq A
  \subseteq F \cup I$.  If $p, q \in P$ say that $q$ \emph{extends}
  $p$ if every set which satisfies $q$ also satisfies $p$.  We must
  construct an shhi set $A$ with upper density at least $1 -
  \epsilon$, and for this it suffices to meet the following
  requirements:
$$N_{2e}: (\exists s \geq e)[\rho_s(A) \geq q_0]$$
$$N_{2e+1} : \mbox {If } \varphi_e \mbox{ is total } \ \& \ 
(\forall a)(\forall b)[a \neq b \rightarrow W_{\varphi_e(a)} \cap
W_{\varphi_e(b)} = \emptyset] \mbox{ then } (\exists a)[W_{f(a)} \cap
A = \emptyset]$$ The result to be proved is an easy consequence of the
following lemma.
\begin{lem} \label{dense}
  For any $p \in P$ and $n \in \omega$ there exists $q \in P$ such
  that $q$ extends $P$ and every set which satisfies $q$ also satisfies
  the requirement $N_n$.
\end{lem}
\begin{proof}
  To prove the lemma, let $p = (F,I)$.  Consider first the case where
  $n = 2e$.  Since $\overline{\rho}(I) > q_0$, there exists $s > e$
  with $\rho_s(I) \geq q_0$.  Let $q = (\hat{F}, \hat{I})$, where
  $\hat{F} = F \cup \{i \in I : i < s\}$, and $\hat{I} = \{i \in I : i
  \geq s\}$.  Then $q \in P$, $q$ extends $p$, and $\rho_s(\hat{F})
  \geq q_0$.  Furthermore, if $A$ satisfies $q$, then $\rho_s(A) \geq
  q_0$ because $A \supseteq \hat{F}$, so $A$ meets $N_n$.

For the case where $n = 2e+1$, we prove the following combinatorial
lemma.  \begin{lem} \label{comb} Suppose the sets $S_0, S_1, \dots$
are pairwise disjoint and $I$ is a set such that $\overline{\rho}(I) >
q_0$.  Then, for all sufficiently large $j$, $\overline{\rho}(I \setminus S_j) > q_0$.
\end{lem}

\begin{proof}
  Assume the result fails, so for infinitely many $j$, we have
  $\overline{\rho} (I \setminus S_j) \leq q_0$.  In fact, we may
  assume without loss of generality that this inequality holds for all
  $j$, since we may replace the sequence of all $S_j$'s by the
  sequence of those $S_j$'s for which it holds.
  Choose rational numbers $q_1, q_2$ such that $q_0 < q_1 < q_2 <
  \overline{\rho}(I)$.  Since $q_2 < \overline{\rho}(I)$, we may
  choose numbers $n_0 < n_1 < \dots$ such that $\rho_{n_i}(I) \geq
  q_2$ for all $i$.  Then we have $$\rho_{n_i}(I) = \rho_{n_i}(I \cap
  S_j) + \rho_{n_i}(I \setminus S_j)$$ for all $i, j$.  Since, for all
  $j$, $\overline{\rho}(I \setminus S_j) \leq q_0 < q_1$, we have that
  $$\rho_{n_i}(I \setminus S_j) \leq q_1$$ for all $j$ and all
  sufficiently large $i$ (dependent on $j$).  It follows that for all
  $j$, if $i$ is sufficiently large, $$\rho_{n_i}(I \cap S_j) =
  \rho_{n_i}(I) - \rho_{n_i}(I \setminus S_j) \geq q_2 - q_1 > 0$$
  Choose $n > (q_2 - q_1)^{-1}$, and then choose $i$ sufficiently
  large that the above inequalities hold for all $j < n$.  Then
  $$\rho_{n_i}(I \cap \cup_{j<n} S_j) = \sum_{j < n} \rho_{n_i}(I \cap
  S_j) \geq n(q_2-q_1) > 1$$ which is absurd because densities can
  never exceed $1$.  This contradiction proves the lemma.  \end{proof}

Now return to the case where $n = 2e+1$ in the proof of Lemma
\ref{dense}, and assume that the hypotheses of $N_{2e+1}$ are
satisfied.  Let $S_k = W_{\varphi_e(k)}$.  Let $p \in P$ be given, and
suppose that $p = (F,I)$.  Since $S_0, S_1, \dots$ are pairwise
disjoint, there are only finitely many $k$ such that $S_k \cap F \neq
\emptyset$.  Hence, by Lemma \ref{comb}, there exists $k$ such that
$S_k \cap F = \emptyset$ and $\overline{\rho}(I \setminus S_k) > q_0$.
Define $q = (F, I \setminus S_k)$.  Then $q \in P$, and $q$ extends $p$.
If $A$ satisfies $q$, then $A \cap S_k = \emptyset$, so $A$ satisfies $R_n$
\end{proof}

The proof of part (2) of Theorem \ref{shhi} is now standard.  Namely,
we inductively choose $p_0, p_1, \dots$ such that each $p_n$ is in
$P$, $p_{n+1}$ extends $p_n$ for all $n$, and every set which
satisfies $p_{n+1}$ meets the requirement $N_n$.  This is possible by
letting $p_0 = (\emptyset, \omega)$ and applying Lemma \ref{dense}.
If $p_n = (F_n, I_n)$, let $A = \cup_n F_n$.  Then $A$ satisfies every
$p_n$ and hence meets every requirement.

\end{proof}

\begin{obs} \label{coh}

  If $A$ is r-cohesive, then $\rho(A) = 0$ and $A$ is
             computable at every density $r < 1$.
\end{obs}

\begin{proof}
   Suppose that $A$ is r-cohesive, and $n > 1$ is given. Since
  the union of all the congruence classes modulo $n$ is $\omega$ and
  $A$ is infinite, some congruence class, say $[i]$, must have
  infinite intersection with $A$.  Then all but finitely many elements
  of $A$ must belong to $[i]$ and all congruence classes $[j]$ with $
  0 \le j < n, j \ne i$ must have finite intersection with $A$ since
  $A$ is r-cohesive.  Let $S_n$ be the union of all the classes $[j]$
  with $j \ne i$.  Then $S_n$ is a computable set of density $1 -
  1/n$, and $S_n \cap A$ is finite.  It follows that $A$ has upper
  density at most $1/n$, and that $A$ is computable at density $1 -
  1/n$.  Since $n$ is arbitrary, the result follows.

\end{proof}

  This   ``approachability'' phenomenon holds  very generally.

\begin{defn} If  $A \subseteq \omega$, the \emph{asymptotic computability bound} of $A$
is 
\newline $ \alpha(A) := sup\{r : A \mbox{ is computable at density } r \}$ .
\end{defn}

\begin{thm}  If $r \in (0,1)$, 
then there is a set $A$ of density $r$ with $\alpha(A) = r$. 
If $r$ is not left-$\Sigma^0_3$ then $A$  is not computable at density $r$.
\end{thm}

\begin{proof}  Let $.b_0 b_1 ...$ be the binary expansion of $r$.
(If $r$ is a sum of finitely many powers of $2$, take the expansion
with infinitely many $1$s.)
In \cite{JS}, Corollary 2.9, it was shown  that the set $D = \bigcup_{ b_i = 1} R_i$
has density $r$. We again take $A = D \cup S$ where $S$ is a simple set
of density $0$.  If $s < r$ we can take enough digits of the expansion
of $r$ so that if  $t = .b_1 ...b_n$  then $s < t < r$. The set $C$
which is the union of the $R_j$ where  $j \le n, b_j \ne 0$
is a computable subset of $A$ of density $t$ so $A$ is computable at density $t$.
Since we can take $t$ arbitrarily close to $r$, it follows that $\alpha(A) \geq r$.
As in Corollary \ref{nohigherdensity}, $A$ is not computable at any density
greater than $r$, so $\alpha(A) \leq r$.  Further, if $r$ is not left-$\Sigma^0_3$ then $A$ is
not computable at density $r$.

\end{proof}.

\section{The minimal pair problem and relative generic computability}

The notion of generic computability can be relativized in the obvious
way.  Specifically, if $A, C \subseteq \omega$, we define $C$ to be
generically $A$-computable if there is a partial $A$-computable
function $\psi$ such that $\psi(n) = C(n)$ for all $n$ in the domain
$D$ of $\psi$ and, further, the domain $D$ has density $1$.  In this
section, we show that if $A,B$ are noncomputable $\Delta^0_2$ sets,
there is a set $C$ such that $C$ is generically $A$-computable and
generically $B$-computable and yet $C$ is not generically computable.  After we
obtained this result, Gregory Igusa \cite{I} greatly strengthened it by showing
that it holds even without the assumption that $A$ and $B$ are
$\Delta^0_2$ sets.  Thus, there are no minimal pairs for relative
generic computability, even though minimal pairs exist in abundance
for relative Turing computability, i.e.  Turing reducibility.  Even
though our result has been superseded by Igusa's subsequent work, we
include it here because the case where $A$ and $B$ are $\Delta^0_2$ is
a major stepping stone towards his remarkable result.

Note that relative generic computability is not transitive, as shown
in \cite{JS}, Section 3.  A stronger transitive notion called ``generic
reducibility'' is defined in Section 4 of \cite{JS}, and studied
further in \cite{I}.  The existence of minimal pairs for generic
reducibility remains open.

The following result is fundamental to our approach.  Recall that
$D_n$ is the finite set with canonical index $n$.  Here in fact it is
important that we use the standard canonical indexing, i.e. $D_0 =
\emptyset$ and, and if $n_1, n_2, \dots, n_k$ are distinct nonnegative
integers and $n = \sum_{i= 1}^k 2^{n_i}$, then $D_n = \{n_1, n_2,
\dots, n_k\}$.

\begin{thm} Suppose that $A$, $B$ are infinite sets such that $A \cup
  B$ is hyperimmune, and let 
$$C = \{n : D_n \cap (A \cup B) \neq \emptyset\}.$$  
Then $C$ is generically $A$-computable and
  generically $B$-computable but not generically computable.
\end{thm}

\begin{proof}
To prove this result, we need the following lemma.

\begin{lem} \label{infd1}  Let $I$ be an infinite set.   Then $\{n: D_n \cap I
\neq \emptyset\}$ has density $1$.
\end{lem}

\begin{proof}   Note first that, for any $m > 0$ and any $b$, the set
$S$ of numbers congruent to $b$ modulo  $m$ has density $1/m$, as can
be shown by an elementary calculation.

Now let $D$ be a nonempty finite set, and let $T = \{n : D_n \cap D =
\emptyset\}$.  We now claim that $\rho(T) = 2^{-|D|}$.   Let $m = \max D$.
By our choice of indexing of finite sets, the elements of $T$ are
exactly the numbers which have a $0$ in places of their binary
expansion corresponding to elements of $D$, so for all $a$, $a \in T$
iff $(a + 2^{m+1}) \in T$.  Hence $T$ is a finite union of residue
classes modulo $2^{m+1}$.  Since each of these residue classes has a
density, $T$ has a density.  To calculate this density, note that if
$k$ is a multiple of $2^{m+1}$, then each of the $2^k$ ways of filling
in the places of the binary expansion corresponding to elements of $D$
occurs equally often in numbers below $k$, so $\rho_k(T) = 2^{-k}$.
Since $T$ has a density and $\rho_k(T) = 2^{-k}$ for infinitely many $k$,
we have that $\rho(T) = 2^{-k}$.

It now follows that, for every $k$, $\{n : D_n \cap I = \emptyset\}$
has upper density at most $2^{-k}$, so this set has density $0$ and its
complement has density $1$.
\end{proof}

We now complete the proof of the theorem.  Recall that $A$ and $B$ are
infinite, $A \cup B$ is hyperimmune, and $C = \{n : D_n \cap (A \cup
B) \neq \emptyset\}$.  To show that $C$ is generically $A$-computable,
define $\psi(n) = 1$ if $D_n \cap A \neq \emptyset$.  Obviously, if
$\psi(n) \downarrow$, then $\psi(n) = 1 = C(n)$, since $D_n \cap A
\neq \emptyset$. Also, the domain of $\psi$ has density $1$ by the
lemma and the assumption that $A$ is infinite.  The proof that $C$ is
generically $B$-computable is the same.  

It remains to show that $C$ is not generically computable.  Suppose
for a contradiction that $C$ were generically computable.  Note that
$\rho(C) = 1$ by the lemma.  If $\psi$ is a computable partial
function which witnesses that $C$ is generically computable, then $T =
\{n : \psi(n) = 1\}$ is a c.e.\ set of density $1$ which is a subset of
$C$.  We now obtain the desired contradiction.   Namely, we show that $A \cup B$ is not
hyperimmune by constructing a strong array $F_0, F_1, \dots$ of pairwise
disjoint finite sets intersecting $A \cup B$.  Suppose that $F_i$ has
been defined for all $i < j$.  Let $m$ exceed all elements of $\cup_{i
  < j} F_i$.  Let $U = \{n : D_n \cap [0, m) = \emptyset\}$.  As shown
in the proof of the lemma, $\rho(U) = 2^{-m}$.  Since $\rho(T) = 1$,
it follows that $T \cap U \neq \emptyset$.  By effective search, one
can find $n \in T \cap U$.  Let $F_j = D_n$.
\end{proof}

\begin{thm} Let $A_0$ and $B_0$ be noncomputable $\Delta^0_2$ sets.  Then
  there is a set $C$ which is both generically $A_0$-computable and
  generically $B_0$-computable but is not generically computable.
\end{thm}

\begin{proof} Note that the family of generically $A_0$-computable sets
  depends only on the degree of $A_0$.  Hence, by the previous
  theorem, it suffices to show that any two nonzero degrees ${\bf a, b
    \leq 0'}$ are \emph{jointly hyperimmune}, meaning that there are sets
$A, B$ of degree ${\bf a, b}$ respectively with $A \cup B$ hyperimmune.
The following lemma is helpful for this.

\begin{lem}   If $A$ is hyperimmune and $B$  is $A$-hyperimmune, then
$A \cup B$ is hyperimmune.
\end{lem}

\begin{proof}   Suppose that $A \cup  B$ is infinite and not hyperimmune.
Then there is a strong array $\{F_n\}$ which witnesses this.   If
$F_n \cap A \neq \emptyset$ for all but finitely many $n$, then we can
conclude that $A$ is not hyperimmune.   Otherwise, there are infinitely
many $n$ such that $F_n \cap A = \emptyset$.   Then the family of such sets
$F_n$ can be made into an $A$-computable array, and this array witnesses
that $B$ is not $A$-hyperimmune.
\end{proof}

If ${\bf a = b}$, then ${\bf a, b}$ are jointly hyperimmune by the
theorem of Miller and Martin \cite{MM} that every nonzero degree below
$\bf 0'$ is hyperimmune.  Otherwise, we may assume without loss of
generality that ${\bf b \not \leq a}$.  In this case we can argue that
$\bf b$ is $\bf a$-hyperimmune.  By the proof of the Miller-Martin
result \cite{MM} there is a function $f$ of degree $\bf b$ such that
every function which $g$ which majorizes $f$ can compute $f$.  Since
${\bf b \not \leq a}$, no $\bf a$-computable function can compute $f$.
By the standard majorization characterization of hyperimmunity,
relativized to ${\bf a}$, it follows that $\bf b$ is ${\bf a}$-hyperimmune.
Hence by the above lemma ${\bf a, b}$ are jointly hyperimmune.
As remarked above, this suffices to prove the theorem.

\end{proof}

\section{Absolute undecidability}

In this section we mention some results on a very strong form of generic
noncomputability introduced by Myasnikov and Rybikov \cite{MR}.
For comparison, recall that a set $A$ is \emph{generically computable} if 
there is a partial computable function which agrees with the characteristic
function of $A$ on its domain and has a domain of density $1$.

\begin{defn}\cite{MR} A set $A \subseteq \omega$ is \emph{absolutely
    undecidable} if every partial computable function agreeing with
  the characteristic function of $A$ on its domain has a domain of
  density $0$.
\end{defn}

If $A$ is absolutely undecidable, then $A$ is not computable at any
$r > 0$.    The next observation shows that the converse fails.

\begin{obs} There are c.e.\ sets which are not computable
at any positive density $r > 0$ but which are not absolutely undecidable.
\end{obs}

\begin{proof}  Let $C$ be any c.e.\ set with lower density $0$ but with
positive upper density.  Let $A = C \cup S$ where $S$ is a simple
set of density $0$. Then $A$ cannot be computable at any density $r > 0$
since a partial algorithm for $A$ whose domain had positive lower density
would have to  answer $n \notin A$ on an infinite set 
which would thus be an infinite c.e.\ set not intersecting $S$.
\end{proof}

As pointed out in \cite{JS}, Observation 2.11, every nonzero Turing
degree contains a set which is not generically computable.  Bienvenu,
Day, and H\"ozl \cite{BDH} have obtained a remarkable generalization of this result.

\begin{theorem} \cite{BDH}  \label{BDH} There exists a Turing functional $\Phi$ such that for every
noncomputable set $A$, $\Phi^A$ is absolutely undecidable and truth-table equivalent
to $A$.   Hence, every nonzero Turing degree contains an absolutely undecidable set.
\end{theorem}

The idea of the proof of \cite{BDH} is to code $A$ into $\Phi^A$ using
an error correcting code (the Hadamard code) with sufficient
redundancy so that, given any partial description of $\Phi^A$ defined
on a set of positive upper density, it is possible to effectively
recover $A$.

In the other direction, we have the following result, which shows that it is impossible to strengthen Theorem \ref{BDH}
by requiring $\Phi^A$ or its complement to be immune.

\begin{theorem} \label{nonimmune} There is a noncomputable set $A$
  such that for every absolutely undecidable set $B \leq_T A$, neither
  $B$ nor $\overline{B}$ is immune.
\end{theorem}

This result is proved by analyzing a version of the construction of a non-computable set of bi-immune free
degree \cite{J69a}.   We omit the details.   Note that Theorem \ref{nonimmune} immediately implies the existence
of a non-zero bi-immune free degree since every bi-immune set is absolutely undecidable.

\end{document}